\documentclass[a4paper,11pt,oneside]{amsart}
\usepackage{mathtools,amssymb,amsthm}
\usepackage[shortlabels]{enumitem}
\usepackage[english]{babel}
\usepackage[utf8]{inputenc}
\usepackage[colorlinks=true, citecolor=cyan, linkcolor=blue, urlcolor=cyan, filecolor=magenta]{hyperref}
\usepackage{setspace, bbm, graphicx}
\usepackage[numbers]{natbib}
\usepackage{geometry}
\setlist[enumerate]{itemsep=0mm}
\usepackage{tikz}
\usetikzlibrary{backgrounds}
\usetikzlibrary{arrows}
\usetikzlibrary{shapes,shapes.geometric,shapes.misc}

\tikzstyle{tikzfig}=[baseline=-0.25em,scale=0.5]

\pgfkeys{/tikz/tikzit fill/.initial=0}
\pgfkeys{/tikz/tikzit draw/.initial=0}
\pgfkeys{/tikz/tikzit shape/.initial=0}
\pgfkeys{/tikz/tikzit category/.initial=0}

\pgfdeclarelayer{edgelayer}
\pgfdeclarelayer{nodelayer}
\pgfsetlayers{background,edgelayer,nodelayer,main}

\tikzstyle{none}=[inner sep=0mm]

\newcommand{\tikzfig}[1]{%
{\tikzstyle{every picture}=[tikzfig]
\IfFileExists{#1.tikz}
  {\input{#1.tikz}}
  {%
    \IfFileExists{./figures/#1.tikz}
      {\input{./figures/#1.tikz}}
      {\tikz[baseline=-0.5em]{\node[draw=red,font=\color{red},fill=red!10!white] {\textit{#1}};}}%
  }}%
}

\newcommand{\ctikzfig}[1]{%
\begin{center}\rm
  \input{#1.tikz}
\end{center}}

\tikzstyle{every loop}=[]

\tikzstyle{gray}=[-, draw=black, fill={rgb,255: red,128; green,128; blue,128}]
\tikzstyle{red}=[-, draw=red]
\tikzstyle{gray_line}=[-, draw={rgb,255: red,170; green,170; blue,170}]
\tikzstyle{arrow}=[<-, draw=black]
\tikzstyle{light gray}=[-, draw=black, fill={rgb,255: red,191; green,191; blue,191}]

\newcommand{\RR}{\mathbb{R}}
\newcommand{\ZZ}{\mathbb{Z}}

\newcommand{\PP}{\mathbb{P}}
\newcommand{\EE}{\mathbb{E}}

\newcommand{\ee}{\varepsilon}

\newcommand{\abs}[1]{\left| #1 \right|}
\renewcommand*{\bf}[1]{\ifmmode\mathbf{#1}\else\textbf{#1}\fi}

\newcommand{\mc}[1]{\mathcal{#1}}

\newcommand{\inn}[1]{\langle #1 \rangle}
\newcommand{\td}[1]{{\widetilde{#1}}}

\newcommand{\one}{\mathbbm{1}}
\newcommand{\FF}{\mathcal{F}}

\newcommand{\XX}{\mathcal{X}}
\newcommand{\rect}[4]{\mathcal{R}^{#1}_{#2}({#3}, {#4})} 
\newcommand{\dect}[4]{\mathcal{D}^{#1}_{#2}({#3}, {#4})} 

\newcommand{\limb}{{\lim_{\beta \to \infty}}}
\newcommand{\tb}{\theta_\beta}
\newcommand{\capa}{\operatorname{cap}_\beta}
\DeclareMathOperator{\opt}{opt}

\newtheorem{theorem}{Theorem}[section]
\newtheorem{lemma}[theorem]{Lemma}

\newtheorem{proposition}[theorem]{Proposition}

\theoremstyle{definition}
\newtheorem*{remark}{Remark}
\newtheorem*{claim}{Claim}

\title[Metastability of the Three-state Potts Model]{Metastability of the Three-state Potts Model with Asymmetrical External Field}
\author{Jeonghyun Ahn}
\address{J. Ahn. Department of Mathematical Sciences, Seoul National University,
Republic of Korea.}
\email{ihenry@snu.ac.kr}

\begin{document}
\maketitle

\begin{abstract}
    In this paper, we explore the metastable behavior of the Glauber dynamics associated with the three-state Potts model with an asymmetrical external field at a low-temperature regime. The model exhibits three monochromatic configurations: a unique stable state and two metastable states with different stability levels. We investigate metastable transitions, which are transitions from each metastable state to the ground state, separately and verify that the two transitions exhibit different behaviors.

    For the metastable state with greater stability, we derive large deviation-type results for metastable transition time, both in terms of probability and expectation. On the other hand, a particularly intriguing phenomenon emerges when starting from the other metastable state: the process may fall into the deep valley of another metastable state with a low probability but remains trapped there for an exponentially long time. We identify specific conditions on the external field under which this rare event contributes to the mean hitting time. To this end, we conduct a detailed analysis of the energy landscape, revealing a sharp saddle configuration analogous to the Ising model.
\end{abstract}

\section{Introduction}

Metastability is a ubiquitous phenomenon observed when a physical system is close to a first-order phase transition. It is widely observed in stochastic systems, including small random perturbations of dynamical systems, condensing interacting particle systems, and spin systems at low temperatures. For a detailed explanation of the history of metastability, we refer to the comprehensive monographs \cite{bovier2015metastability, olivieri2005large}. In the analysis of this phenomenon, mathematical frameworks to describe the physical system have been established. Typically, a system is modeled by a mathematical function known as Hamiltonian $H$, and evolution is followed by specific rules described in \eqref{eq:CTMC} known as the Metropolis-type Glauber dynamics. In essence, the dynamic tends to reduce the Hamiltonian, while the reverse direction occurs with an exponentially low probability.

We investigate the metastable behavior of the \emph{Potts model}, a generalization of the well-known \emph{Ising model}. Starting from \cite{baxter1982critical, baxter1973potts}, its dynamics have been extensively studied. The mean-field version of the Potts model has been discussed in \cite{costeniuc2005complete, ellis1990limit, ellis1992limit, gandolfo2010limit}. Additionally, the Potts model with zero external field has been explored in \cite{bet2021critical,kim2021metastability, kim_2022, nardi2019tunneling}. Various other generalizations exist, such as degenerate external fields \cite{Bet_2022, bet2022metastability} and general interaction constants \cite{kim2022potts}.

In this paper, we focus on the three-state Potts model with an asymmetrical external field, defined on a two-dimensional lattice graph with a periodic boundary condition.

Recall that the Ising model on a finite graph $\Lambda = (V, E)$ is defined by a Hamiltonian function
\[
    H^{\mathrm{Ising}}(\sigma) = - J\sum_{(x, y) \in E} \one_{\{\sigma(x) = \sigma(y)\}} - h\sum_{x \in V} \one_{\{\sigma(x) = 1\}}
\]
where $\sigma \in \{1, -1\}^V$ is a spin configuration on $\Lambda$. The dynamics consist of two parameters, the \emph{interaction constant} $J > 0$ and the \emph{external field} $h > 0$. The model is said to be \emph{ferromagnetic} since the spins favor being aligned. 

The Potts model is a generalization of the Ising model to an arbitrary number of spins, denoted by $q$. There are various ways of generalization, depending on the form of Hamiltonian $H$. In this paper, we fix the interaction constant but assign an asymmetrical external field for each spin. Namely, define the Hamiltonian $H$ by
\begin{equation}
    \label{eq:hamiltonian}
    H(\sigma) = -J\sum_{(x, y) \in E} \one_{\{\sigma(x) = \sigma(y)\}} - \sum_{x \in V} h_{\sigma(x)}
\end{equation}
for $\sigma \in \{1, \ldots, q\}^V$. Each spin $i$ has a different value of external field denoted by $h_i$. This differentiates the energy of $q$ monochromatic states, making exactly one into the most stable state. Furthermore, we assume $q = 3$ to avoid a complex energy landscape.

\begin{figure}
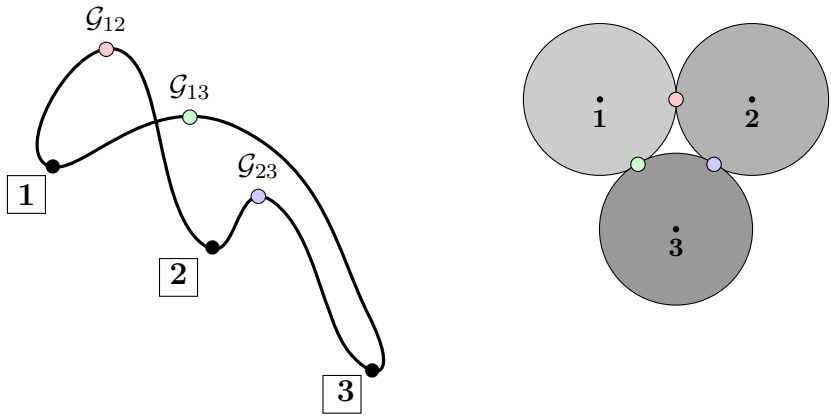

\centering
\ctikzfig{energy}
\caption{\label{fig:energy} Schematic picture of the energy landscape of the three-state Potts model with an asymmetrical external field. An Ising-type reference path connects each pair of monochromatic states.}
\end{figure}

Our goal is to investigate the metastable behavior of the three-state Potts model. Precisely, we show the following results:
\begin{itemize}[leftmargin=*]
    \item large deviation-type result, using the pathwise approach described in \cite{Nardi_2015},
    \item the Eyring--Kramers law for transition between monochromatic configurations, applying potential theory,
    \item the Markov chain model reduction, introduced in \cite{Beltr_n_2010}.
\end{itemize}

We start by describing the energy landscape. Denote by $\bf i$ the monochromatic configuration with spin $i$. Without loss of generality, let $h_3 > h_2 > h_1 > 0$.  In this scenario, we have a unique stable state \bf3, and two metastable states \bf1 and \bf2, however, with different stability levels. To see why they differ, it is helpful to examine a simplified model. For a while, let us only consider configurations with at most two types of spin. Then the model is a combination of Ising models between two different spins. Hence, the reference path of the original Ising model gives energy-minimizing paths between \bf1, \bf2, and \bf3. In fact, we show that these paths are still the energy-minimizing paths in our Potts model. Hence, the energy barrier of transitions $\bf1\to\bf3$ and $\bf2\to\bf3$ is $\Gamma^\ast(J, h_3 - h_1)$ and $\Gamma^\ast(J, h_3 - h_2)$, respectively. Here, $\Gamma^\ast(J, h)$ denotes the energy barrier of the Ising model with parameters $J$ and $h$ previously discovered in \cite{neves1991critical}. This quantity strictly decreases on $h \in (0, J)$, thus \bf2 has a larger stability level than \bf 1. Gathering these observations, the overall picture of the energy landscape is depicted in Figure \ref{fig:energy}. We also remark that the transition is \emph{sharp} in the sense that the transition goes through a single critical configuration.

Our next focus is the Eyring--Kramers law. Estimation of the mean transition time from \bf2 to \bf3 is straightforward; since \bf2 has the maximal stability, there is no chance for the process to be trapped in a valley other than \bf2 or \bf3. As soon as the process escapes the valley of \bf2 and passes the critical configuration, the process is likely to fall into \bf3. This metastable behavior enables the estimation of the transition time. A technical tool known as \emph{potential theory} aids this estimation. In brief, it suffices to estimate a capacity between \bf2 and \bf3, which readily follows from the sharp saddle configuration. In contrast, the transition from \bf1 to \bf3 has a different behavior since \bf1 has less stability than \bf2. In view of the typical path, the transition takes place away from \bf2. However, the process may fall into \bf2 with low probability, and it takes a long time to escape the deep valley of \bf2. This may contribute to the mean hitting time $\EE_{\bf1}^\beta[\tau_{\bf3}]$ of the transition $\bf1 \to \bf3$. We show in Theorem \ref{thm:EKinf} that if $2h_2 > h_1 + h_3$, then the mean hitting time is dominated by the transition $\bf1 \to \bf2 \to \bf3$, having a larger time scale than a typical transition.

Finally, we show the Markov chain model reduction. Since \bf1 and \bf2 have different depths of valley, say $\mc V_{\bf1}$ and $\mc V_{\bf2}$, we may speed up the Markov chain by two different time scales. When speeded-up by $e^{\beta\mc V_{\bf1}}$, only the transition $\bf1 \to \bf3$ is observed, and \bf2 remains isolated. When speeded-up by $e^{\beta\mc V_{\bf2}}$, there is a one way transition $\bf2 \to \bf3$, and the state \bf1 is not observed. Our model is one example in which multiple time scales appear in a Markov chain model reduction. See \cite{Bianchi_2017, kim2021second, kim2023hierarchical, landim2023metastability} for other examples.

\paragraph{Mathematical Obstacles}
The main obstacle lies in the analysis of the energy landscape, including the computation of communication heights and stability levels for various configurations. In particular, the situation becomes complex when all three spins are involved in the configuration. To resolve this issue, we employ the \emph{projection operator} denoted as $\mc P_{ij}$ on the configuration space, which is introduced in \cite{kim2022potts}. This operator $\mathcal{P}_{ij}$ transforms $i$ spins into $j$ spins, resulting in configurations with at most two spins. These reduced configurations can then be analyzed using well-established knowledge about the Ising model.

An additional obstacle arises when bounding the stability level of configurations other than \bf1, \bf2, and \bf3. A particularly complex scenario occurs when a cluster of spin 3 is surrounded by both spins 1 and 2. This appears in the proof of Theorem \ref{thm:stab2} in Subsection \ref{subsec:stab}, specifically in \textbf{(Case 2)} of the proof. To address this issue, we begin by finding configurations where the stability level can be bounded through local changes. For instance, in Figure \ref{fig:col2}, we can either update spins 1 and 2 to spins 3 along the edge or do the opposite, thereby reducing the overall energy. By eliminating such cases, we can narrow the scenario to the worst case, which is depicted in Figure \ref{fig:windmill}. To see how this specific case is managed, we refer to \textbf{(Case 2-2)} in the proof of Theorem \ref{thm:stab2}, which can be found in Subsection \ref{subsec:stab}.

We also highlight the use of Proposition \ref{prop:neg} which states the so-called \emph{negligibility} condition of the Markov chain model reduction discussed in \cite{Beltr_n_2010}. It is a common belief that we may ignore every shallow valley with a depth less than that of metastable states. However, to our best knowledge, the negligibility is typically proven independently for each model. We remark that Proposition \ref{prop:neg} gives the proof of negligibility in the Metropolis dynamics in the most general form. We also remark that this applies to models with multiple speed-up scales.

\paragraph{Relavent Works}
We end by introducing other models which have relevance to our setting. In \cite{Bet_2022,bet2022metastability}, the authors consider the degenerate Potts model, which is the case $h_1 \neq h_2 = h_3 = \cdots = h_q$ in our notation. In this setting, there are either $q-1$ metastable states and one stable state, or one metastable state and $q-1$ stable states. The transition from a metastable to a stable state is studied by classifying critical configurations and the tube of typical paths. In the model with $q-1$ stable states in \cite{Bet_2022}, a transition between stable states is also studied, which has an analogy to the zero external field case.

In \cite{Landim_2016}, the authors investigate the metastability of the two-dimensional Blume--Capel model. This model has two metastable states and a unique stable state. A notable behavior of this model is observed on a transition between two metastable states. With a low probability, the process deviates from a typical path and falls into a stable state for an exponentially long time. This rare event dominates the mean hitting time, which is exactly what happens in our model during a transition $\bf1 \to \bf3$ under the condition $2h_2 > h_1 + h_3$.

Finally, we refer to \cite{kim2022potts} which deals with a similar setting to this paper. It works on a three-state Potts model with zero external field and general interaction constants $J_{ij}$ between arbitrary two spins $i$ and $j$. Some analogous tools are used in this paper, such as the projection operator defined in Lemma \ref{lem:proj}.

\paragraph{Outline}
The paper is organized as follows. In Section 2 we define the three-state Potts model with an asymmetrical external field, then present our main results. In Section 3 we briefly recall the notations and preliminaries on the pathwise approach. We prove the main results related to the pathwise approach. Finally, in Section 4 we prove the remaining main results using potential theory. 

\section{Model Description and Main Results}
\label{sec:model}

Let $G = (V, E)$ be a $K \times L$ rectangular lattice graph with a periodic boundary condition. Precisely, the edge set includes the pair of vertices lying on the opposite side of the boundary. On each vertex, we assign one of the three spins 1, 2, and 3. Gathering all possible configurations, we have a state space denoted by $\mathcal{X} = \{1, 2, 3\}^V$. The Hamiltonian function was previously defined in \eqref{eq:hamiltonian}, but we assume $J = 1$ for simplicity. Writing again, the Hamiltonian function $H: \XX \to \RR$ is defined as
\[
    H(\sigma) = -\sum_{\{x,y\} \in E} \one_{\{\sigma(x) = \sigma(y)\}} - \sum_{x \in V} h_{\sigma(x)}.
\]
Here, $\sigma(x)$ denotes a spin of configuration $\sigma \in \XX$ at site $x \in V$. Without loss of generality, assume $h_3 > h_2 > h_1 > 0$. This gives asymmetric stability on each spin. Also, we assume $h_3 < 1$ so that the spin itself has less effect on the energy than the interaction between adjacent cells.

We impose two assumptions analogous to that of the Ising model, which appears in \cite{neves1991critical} as the \emph{standard case}.

\noindent
\textbf{Assumption A}. $K, L > \frac{3}{h_3 - h_2}, \frac{3}{h_2 - h_1}$.

\vspace{1mm}
\noindent
\textbf{Assumption B}. $a(h_2 - h_1) + b(h_3 - h_2) \notin \ZZ$ for $(a, b) \in \ZZ^2 \backslash \{(0, 0)\}$.

Assumption A is required to ensure the existence of a critical droplet. Assumption B corresponds to the condition $2 / h \notin \ZZ$ of the original Ising model, where $h$ is the external field of the Ising model.

We say that two states $\sigma, \eta \in \XX$ are adjacent if they differ by exactly one site. The dynamics is defined as a continuous time Markov chain $\{X_\beta(t)\}_{t\ge 0}$ on $\XX$ with a transition rate
\begin{equation}
    \label{eq:CTMC}
    c_\beta(\sigma, \eta) = \begin{cases}
        e^{{-\beta[H(\eta) - H(\sigma)]}_+} & \text{if }\sigma \sim \eta \\
        0 & \text{otherwise}
    \end{cases}
\end{equation}
where $[t]_+ = \max(t, 0)$ and the parameter $\beta > 0$ represents the inverse temperature. The Markov chain has a unique invariant measure called the \emph{Gibbs measure}, given as
\[
    \mu_\beta(\sigma) = \frac{1}{Z_\beta} e^{-\beta H(\sigma)}, \quad Z_\beta = \sum_{\sigma \in \XX} e^{-\beta H(\sigma)}.
\]
We see that
\[
    \mu_\beta(\sigma) c_\beta(\sigma, \eta) = \mu_\beta(\eta) c_\beta(\eta, \sigma) = \frac{1}{Z_\beta} e^{-\beta \max (H(\sigma), H(\eta))},
\]
so $\sigma_\beta(t)$ is a reversible Markov chain. Denote 
by $\PP^\beta_\sigma$ and $\EE^\beta_\sigma$ the law and the expectation, respectively, of the Markov chain $\{X_\beta(t)\}$ starting from $\sigma \in \XX$. We omit the subscript or superscript $\beta$ if it is clear from the context.

We define some notations that capture essential information from the energy landscape. First, a \emph{path} from $\sigma \in \XX$ to $\eta \in \XX$ is a sequence of states $\{\omega_t\}_{t=0}^T$ such that
\[
    \omega_0 = \sigma,\, \omega_T = \eta,\, c_\beta(\sigma_t, \sigma_{t+1}) > 0
\]
for all $0 \le t \le T - 1$. This is simply a possible way to move from $\sigma$ to $\eta$. 
Denote all possible paths from $\sigma$ to $\eta$ by $\Omega_{\sigma, \eta}$. For simplicity, we sometimes use the notation $\gamma: \tau \to \eta$ to indicate $\gamma \in \Omega_{\sigma, \eta}$. Define the \emph{communication height} between $\sigma, \eta \in \XX$ by
\[
    \Phi(\sigma, \eta) = \min_{(\omega_t)_{t=0}^T \in \Omega_{\sigma, \eta}}\max_{0 \le t \le T} H(\omega_t).
\]
Define the set of \emph{optimal path} between $\sigma, \eta \in \XX$ by the collection of paths attaining the minimal energy, that is,
\[
    \Omega_{\sigma, \eta}^{\opt} = \{\omega \in \Omega_{\sigma, \eta}: \, \sup_{\xi \in \omega} H(\xi) = \Phi(\sigma, \eta)\}.
\]
Accordingly, we define the communication height between disjoint sets $A, B \subseteq \XX$ by
\[
    \Phi(A, B) = \min_{\sigma \in A, \sigma' \in B} \Phi(\sigma, \sigma')
\]
and the set of optimal paths by
\[
    \Omega_{A,B}^{\opt} = \{\omega \in \Omega_{\sigma, \sigma'}^{\opt}: \, \sigma \in A, \, \sigma' \in B, \, \Phi(\sigma, \sigma') = \Phi(A, B) \}.
\]
Also, define the \emph{energy barrier} from $\sigma \in \XX$ to $A \subseteq \XX$ by
\[
    \Gamma(\sigma, A) = \Phi(\sigma, A) - H(\sigma).
\]
Next, define the \emph{stability level} $\mathcal{V}_\sigma$ of a state $\sigma \in \XX$ by
\[
    \mathcal{V}_\sigma \coloneqq \Gamma(\sigma, I_\sigma),
\]
where
\[
    I_\sigma \coloneqq \{\eta \in \XX: \, H(\eta) < H(\sigma)\}.
\]
Let $\mathcal{V}_\sigma = \infty$ if $I_\sigma = \emptyset$.

Finally, we define the \emph{minimal saddle} between $A, B \subseteq \XX$ by
\[
    \mc S(A, B) \coloneqq \{ \sigma \in \XX : \, \sigma \in \mathrm{argmax}_\omega H \text{ for some } \omega \in \Omega_{A, B}^{\opt} \}.
\]
A set $G \subseteq \XX$ is a gate between disjoint subsets $A, B \subseteq \XX$ if $G \subseteq S(A, B)$ and $\omega \cap G \neq \emptyset$ for all $\omega \in \Omega_{A, B}^{\opt}$. Moreover, we call a gate between $A$ and $B$ is \emph{minimal} if it is minimal with respect to set inclusion amongst all gates between $A$ and $B$. Denote by $\mc G(A, B)$ the union of all minimal gates between $A$ and $B$.

While computing the energy barriers in the model, we will use an auxiliary function $f: (0, 1) \to \RR$ defined by
\[
    f(h) \coloneqq 4\ell_c - h(\ell_c(\ell_c - 1) + 1), \quad \ell_c = \left\lceil \frac{2}{h} \right\rceil.
\]
Recall from \cite{neves1991critical} that $f(h)$ is the energy barrier of the Ising model with external field $h$, and $\ell_c$ appears in the size of the critical droplet. For $1 \le i < j \le 3$, denote
\begin{equation}
    \label{eq:crit}
    \Gamma_{ij}^\star = f(h_j - h_i), \quad \ell_c^{ij} = \left\lceil \frac{2}{h_j - h_i} \right\rceil
\end{equation}
and define $\Phi_{ij}^\star = \Gamma_{ij}^\star + H(\bf i)$ for $1 \le i \neq j \le 3$. Note that the quantities with star, such as $\Gamma_{ij}^\star$ and $\Phi_{ij}^\star$, are directly computable from the given parameters $h_i$'s, and they may differ from $\Gamma(\bf i, \bf j)$ and $\Phi(\bf i, \bf j)$, respectively.

\subsection{Results on energy landscape}
\label{subsec:1}
We start by giving an overall picture of the dynamics. Let \bf1, \bf2, and \bf3 denote the configuration with constant spin 1, 2, and 3, respectively. The dynamics exhibit three deep valleys, each having \bf1, \bf2, and \bf3 as their bottom. Clearly, \bf3 is the most stable state. In terms of stability levels, the subsequent most stable state is \bf2, followed by \bf1. Furthermore, there exists an Ising-type reference path between each state. When these are combined, it results in the illustration shown in Figure \ref{fig:energy}.

In this subsection, we justify Figure \ref{fig:energy}.

\begin{theorem}[Energy barrier]
    \label{thm:gamma}
    The communication height between \bf1, \bf2, \bf3 is given by
    \[
        \Gamma(\bf1,\bf3) = \Gamma(\bf1, \bf2) = \Gamma_{13}^\star, \quad \Gamma(\bf2, \bf3) = \Gamma_{23}^\star.
    \]
\end{theorem}
In other words, the communication level of $\mathbf{i}$ and $\mathbf{j}$ is equal to that of the Ising model, except for between \textbf{1} and \textbf{2}. Since $\Gamma_{13}^\star < \Gamma_{23}^\star$, \bf1 has lower energy barrier than \bf2 with respect to \bf3. 

\begin{theorem}[Stability level of \textbf{1}, \textbf{2}]
\label{thm:stab}
    The stability level of \textbf{1} and \textbf{2} is
    \[
        \mathcal{V}_{\mathbf{1}} = \Gamma_{13}^\star, \quad \mathcal{V}_{\mathbf{2}} = \Gamma_{23}^\star
    \]
\end{theorem}
Note that Theorem \ref{thm:gamma} provides an upper bound of the stability level. The proof is discussed in Subsection \ref{subsec:land}.

The next theorem shows that \textbf{2} has the largest stability. Moreover, \textbf{1} has the second-largest stability.
\begin{theorem}[Stability level]
\label{thm:stab2}
    For any state $\sigma \neq \bf1, \bf2, \bf3$, we have
    \[
        \mc V_\sigma < \mc V_{\bf1} < \mc V_{\bf2}.
    \]
\end{theorem}

Finally, we determine minimal gates between \bf1, \bf2, and \bf3. Let $\rect i j \ell m$ be the collection of configurations with $i$-spin $\ell \times m$ rectangle in a sea of $j$-spin, and $\dect i j \ell m$ obtained from $\rect i j \ell m$ by attaching a single $i$-spin on the length $\ell$ side of the rectangle.
\begin{theorem}[Minimal gates]
\label{thm:gate}
    Let
    \begin{align*}
        \mc G_{13} = \dect 3 1 {\ell_c^{13}+1}{\ell_c^{13}} , \quad \mc G_{23} = \dect 3 2 {\ell_c^{23}+1}{\ell_c^{23}},
    \end{align*}
    which is a minimal gate of the Ising model between \bf1 and \bf3, \bf2 and \bf3, respectively. Then $\mc G_{13}$ is the minimal gate for transitions $\bf1 \to \bf3$, $\bf1 \to \bf2$, and $\bf1 \to \bf2 \cup \bf3$. Also, $\mc G_{23}$ is the minimal gate for transitions $\bf2 \to \bf3$ and $\bf2 \to \bf1 \cup \bf3$. 
\end{theorem}
The proofs of Theorem \ref{thm:stab2} and \ref{thm:gate} are provided in Subsection \ref{subsec:stab}.

\subsection{Large deviation-type results}
\label{subsec:2}

In this section, we give an exponential scale estimate on a hitting time of \bf3. For a subset $A \subseteq \XX$, we denote by $\tau_A$ the hitting time of $A$ and by $\tau_A^+$ the return time to $A$:
\begin{align*}
    \tau_A & \coloneqq \inf\left\{t \ge 0: X(t) \in A\right\}, \\
    \tau_A^+ & \coloneqq \inf\left\{t > 0 : \, X(t) \in A, \, X(s) \neq X(0) \text{ for some } 0 < s < t\right\}.
\end{align*}

\begin{theorem}[Large Deviation-type result]
    \label{thm:LDP} For any $\ee > 0$, we have
    \begin{align}
        \lim_{\beta\to\infty} \PP_{\bf2}^\beta[e^{\beta(\Gamma_{23}^\star-\ee)} < \tau_{\bf3} < e^{\beta(\Gamma_{23}^\star + \ee)}] = 1 \label{eq:LDP1} \\
        \lim_{\beta\to\infty} \PP_{\bf1}^\beta[e^{\beta(\Gamma_{13}^\star-\ee)} < \tau_{\bf3} < e^{\beta(\Gamma_{13}^\star + \ee)}] = 1 \label{eq:LDP2}
    \end{align}
    Moreover, we have
    \[
        \lim_{\beta\to\infty} \frac{1}{\beta}\log\EE_{\mathbf{2}}^\beta[\tau_{\mathbf{3}}] = \Gamma_{23}^\star
    \]
    and a convergence of distribution under $\PP_{\bf2}^\beta$
    \[
        \frac{\tau_{\bf3}}{\EE_{\bf2}^\beta[\tau_{\bf3}]} \xrightarrow{d} \exp(1)
    \]
    where $\exp(1)$ denotes the exponential random variable with a mean value 1.
\end{theorem}
The underlying argument is based on a pathwise approach, which can be found in \cite{Nardi_2015}. Since \textbf{2} is the most metastable state, the energy barrier $\Gamma_{23}^\star$ directly induces an exponential scale of the hitting time from \textbf{2} to \textbf{3} as well as the distribution.

In contrast to $\PP_{\bf2}^\beta$, the result on $\PP_{\bf1}^\beta$ needs much more detailed knowledge of the energy landscape. In view of \cite[Corollary 3.16]{Nardi_2015}, we need to say that the dynamics avoid a valley deeper than $\mc V_{\bf1}$. This can be done by determining a tube of typical paths from \bf1 to \bf3, which is a region where the dynamics stay with probability close to 1. We saw in Theorem \ref{thm:stab2} that the only valley deeper than $\mc V_{\bf1}$ is that of \bf2, so it suffices to show that \bf2 lies outside of the tube. We review various notions of pathwise approach in Subsection \ref{subsec:pathwise}. 

The tube of typical paths also implies the following proposition.
\begin{proposition}
\label{prop:pot_zero}
    We have that
    \[
        \limb \PP_{\bf1}^\beta[\tau_{\bf2} < \tau_{\bf3}] = 0.
    \]
\end{proposition}
The proof of Theorem \ref{thm:LDP} and Proposition \ref{prop:pot_zero} is given in Subsection \ref{subsec:path_proof}.

\subsection{Eyring--Kramers law}
\label{subsec:3}

We are further interested in the leading coefficient of the exponent term of the mean hitting time, commonly referred to as the \emph{prefactor}. The potential theory is applicable to achieve this. Indeed, the estimation of capacity yields the prefactor, and the capacity can be obtained by examining the saddle structure. We saw in Theorem \ref{thm:gate} that the saddle structure between $\bf 1$ or $\bf 2$ and $\bf 3$ is the same as that of the Ising model. This will lead us to the computation of the prefactor.

We adopt big-$O$ and little-$o$ notation on $\beta$ where $\beta$ goes to infinity. Namely, for a function $f:\RR^+ \to \RR$ on $\beta$, we denote $f(\beta) = O_\beta(1)$ if there exists constants $\beta_0 > 0$ and $C > 0$ such that
\[
    \abs{f(\beta)} \le C
\]
for all $\beta \ge \beta_0$. We denote $f(\beta) = o_\beta(1)$ if
\[
    \limb f(\beta) = 0.
\]
\begin{theorem}[Eyring--Kramers law]
    As $\beta$ goes to infinity, we have
    \label{thm:EK}
    \[
        \EE_{\bf2}^\beta[\tau_3] = \kappa_2 e^{\beta\Gamma_{23}^\star}(1 + o_\beta(1)), \quad \kappa_2 = \frac{3}{4(2\ell_c^{23} - 1)}\frac{1}{\abs{\XX}}.
    \]
\end{theorem}
However, the estimation of $\EE_{\mathbf{1}}^\beta[\tau_{\mathbf{3}}]$ is problematic. This is because the process could fall into \textbf{2} with low probability, but it takes a long time to escape the deep valley of \textbf{2}. Note that this behavior is not captured in the Markov chain model reduction discussed in Subsection \ref{subsec:mcmr}.
\begin{theorem}
    \label{thm:EKinf}
    If $\Gamma_{12}^\star < \Gamma_{23}^\star$, or equivalently $2h_2 > h_1 + h_3$, then
    \[
        \limb e^{-\beta\Gamma_{13}^\star}\EE_{\bf1}^\beta[\tau_{\bf3}] = \infty.
    \]
\end{theorem}
Comparing this result to \eqref{eq:LDP2} of Theorem \ref{thm:LDP}, we indeed see that the rare event falling to \bf2 has a significant contribution to the expectation. Theorems \ref{thm:EK} and \ref{thm:EKinf} are proved in Subsections \ref{subsec:4.2} and \ref{subsec:4.4}, respectively.

\subsection{Markov chain model reduction}
\label{subsec:4}
\label{subsec:mcmr}

In \cite{Beltr_n_2010}, a new approach to describe a metastable behavior was introduced. Speeding up the process by an appropriate time scale, the Markov chain converges to a reduced Markov chain. This methodology is called \emph{Markov chain model reduction}.

In our model, there are two types of speeded-up Markov chains in view of Theorems \ref{thm:LDP} and \ref{thm:EK}. Let
\[
    \tb^1 = \kappa_1 e^{\beta\Gamma_{13}^\star},\quad \tb^2 = \kappa_2 e^{\beta\Gamma_{23}^\star}.
\]
where
\[
    \kappa_1 = \frac{3}{4(2\ell_c^{13} - 1)}\frac{1}{\abs{\XX}},\quad \kappa_2 = \frac{3}{4(2\ell_c^{23} - 1)}\frac{1}{\abs{\XX}}.
\]
Let $\Phi: \XX \to \{\bf1, \bf2, \bf3, [\beta]\}$ be a projection operator defined by identity map on $\{\bf1, \bf2, \bf3\}$, and $[\beta]$ elsewhere. In a sense of \cite[Definition 2.2]{Beltr_n_2010}, the speeded-up chain has the following model reduction.

\begin{theorem}[Markov chain model reduction]
    \label{thm:MCMR1}
    The speeded-up hidden Markov chain $\td X_\beta(t) = \Phi(X_\beta(\tb^1 t))$ converges to the Markov chain on $\{\bf1, \bf2, \bf3\}$ with a jump rate
    \[
        r(\bf i, \bf j) = \begin{cases}
            1 & (\bf i, \bf j) = (\bf1, \bf3) \\
            0 & otherwise.
        \end{cases}
    \]
    Moreover, the time spent outside of $\{\bf1, \bf2, \bf3\}$ is negligible; that is,
    \[
        \limb \EE_{\xi}^\beta \left[\int_0^t \one_{\left\{X(s\tb^1) \notin \{\bf1,\bf2,\bf3\}\right\}} ds \right] = 0
    \]
    for every $\xi \in \XX$ and $t > 0$.
\end{theorem}
A similar result holds for speed-up with rate $\tb^2$. But the crucial difference is that the state \bf1 is not observed.
\begin{theorem}[Markov chain model reduction]
    \label{thm:MCMR2}
    The speeded-up hidden Markov chain $\td X_\beta(t) = \Phi(X_\beta(\tb^2 t))$ converges to the Markov chain on $\{\bf2, \bf3\}$ with a jump rate
    \[
        r(\bf2, \bf3) = 1, \quad r(\bf3, \bf2) = 0.
    \]
    Moreover, the time spent outside of $\{\bf2, \bf3\}$ is negligible; that is,
    \[
        \limb \EE_{\xi}^\beta \left[\int_0^t \one_{\left\{X(s\tb^2) \notin \{\bf2, \bf3\}\right\}} ds \right] = 0
    \]
    for every $\xi \in \XX$ and $t > 0$.
\end{theorem}
Note that the projection $\Phi$ is a projection to $\{\bf1, \bf2, \bf3\}$ for both propositions. This means that \bf1 is not observed when the process is speeded up by $\tb^2$. The proof of Theorems \ref{thm:MCMR1} and \ref{thm:MCMR2} is discussed in Subsection \ref{subsec:4.3}.

\begin{remark}
    In this paper, we focus on the case $q = 3$. The reason we consider the three-spin case only is because the energy landscape gets more complicated when $q \ge 4$. To be precise, the difficulty lies in proving Theorem \ref{thm:stab2}. In the proof we present, there are some exceptional cases as depicted in Figure \ref{fig:windmill}-(a) that we have to handle one by one. There are only a few exceptional cases when $q=3$, but the possibility increases exponentially as $q$ grows, making the model unwieldy.
\end{remark}

\section{Pathwise Approaches and Proofs}
The remaining part of the paper is devoted to the proofs of our main results. In the proofs, two methodologies are mainly used; the pathwise approach and the potential theory. We divide our main results into two groups based on the methodology used and prove them separately in the following sections.

In this section, we briefly recall notations and main results of pathwise approaches. Using this, we prove the results in Subsection \ref{subsec:1} and \ref{subsec:2}.
\subsection{Notations and preliminaries}
\label{subsec:pathwise}
The definitions and results of the pathwise approach are mainly from \cite{Nardi_2015}. For a nonempty subset $A \subseteq \XX$, define the \emph{bottom} of $A$ as
\[
    \FF(A) = \{\sigma \in A :\, H(\sigma) = \min_{\eta \in A} H(\eta)\}.
\]
Also, define its \emph{external boundary} by
\[
    \partial A = \{\sigma \notin A: \exists \eta \in A, \, c_\beta(\eta, \sigma) > 0\}.
\]
A nonempty subset $C \subseteq \XX$ is called a \emph{cycle} if it is either a singleton or it is a connected set such that
\[
    \max_{\eta \in C}H(\eta) < H(\FF(\partial C)).
\]
The depth of a cycle $C$ is defined as
\[
    \Gamma(C) = \left[H(\FF(\partial C)) - H(\FF(C))\right]_+.
\]
The following properties are known about the exit distribution of a cycle. The proof can be found in \cite[Theorem 6.23, Corollary 6.25]{olivieri2005large}.
\begin{proposition}
The following statements hold.
    \label{prop:exit}
    \,
    \begin{enumerate}[(a), leftmargin=*]
        \item There exists $k > 0$ such that for all $\beta$ sufficiently large,
        \[
            \sup_{\eta\in C}\PP_\eta^\beta\left[X_\beta\left(\tau_{\partial C}\right)\notin \FF(\partial C)\right] \le e^{-k\beta}.
        \]
        \item Let $\sigma \in \partial C$. For any $\ee > 0$, for all $\beta$ sufficiently large,
        \[
            \inf_{\eta\in C}\PP_\eta^\beta\left[X_\beta(\tau_{\partial C}) = \sigma \right] \ge e^{-\beta(H(\sigma) - H(\FF(C)) + \ee)}.
        \]
    \end{enumerate}
\end{proposition}
For a nonempty subset $A \subseteq \XX$ and $\eta \in \XX \backslash A$, define the initial cycle $C_A(\eta)$ by
\[
    C_A(\eta) = \{\eta\} \cup \{\sigma \in \XX: \Phi(\eta, \sigma) < \Phi(\eta, A)\}.
\]
In fact, the initial cycle is the maximal cycle containing $x$ and disjoint to $A$. For fixed $A \subseteq \XX$, we may observe that initial cycles are either disjoint or are the same. This gives a partition of $\XX \setminus A$ into maximal cycles. We denote it as
\[
    \mc M (\XX \backslash A ) = \{C_A(\eta)\}_{\eta \in \XX \backslash A}.
\]

Now we define a typical path from $\eta$ and $A$. This is done by determining which transitions between maximal cycles are the most likely ones. It will appear that the typical path provides a bound for hitting time from $\eta$ to $A$.

A \emph{cycle-path} is a finite sequence of cycles $C_1, \ldots, C_m$ such that
\[
    C_i \cap C_{i+1} = \emptyset , \quad \partial C_i \cap C_{i+1} \neq \emptyset
\]
for $i = 1, \ldots, m - 1$. Note that given a path $\gamma = (\gamma_i)_{i=1}^m$ from $x$ to $A$, there is a corresponding cycle-path $(C_A(\gamma_i))_{i=1}^m$.

We say that a cycle-path $(C_1, \ldots, C_m)$ is \emph{connected via typical jumps to $A$}, or simply \emph{vtj-connected} to $A$ if
\[
    \FF(\partial C_i) \cap C_{i+1} \neq \emptyset, \quad \FF(\partial C_m) \cap A \neq \emptyset.
\]
A path is called \emph{typical} if its corresponding cycle-path is vtj-connected.
A union of all vtj-connected cycle-path from $\eta$ to $A$ is called the \emph{typical tube} $T_A(\eta)$. It is called typical in the sense of the next lemma, which is found in \cite[Lemma 3.13]{Nardi_2015}.
\begin{lemma}
    \label{lem:tube}
    Let $\eta \in \XX \setminus A$. There exists $k > 0$ such that for $\beta$ sufficiently large,
    \[
        \PP_\eta^\beta\left[\tau_{\partial T_A(\eta)} \le \tau_A \right] \le e^{-k\beta}.
    \]
\end{lemma}

\subsection{Energy landscape}
\label{subsec:land}
We begin by introducing the \emph{projection operator}, originally introduced in \cite{kim2022potts}, which will be useful when examining the energy landscape.
\begin{lemma}[Projection]
\label{lem:proj}
    Let $\sigma \in \mathcal{X}$ be a configuration. Let $\mathcal{P}_{ij}\sigma$ be the configuration obtained by substituting $i$ spins in $\sigma$ to $j$ spins. Then the following inequality holds.
    \begin{align*}
        H(\mathcal{P}_{12}\sigma) \le H(\sigma), \\
        H(\mathcal{P}_{13}\sigma) \le H(\sigma), \\
        H(\mathcal{P}_{23}\sigma) \le H(\sigma).
    \end{align*}
    Moreover, each equality holds if and only if $\mathcal{P}_{ij}\sigma = \sigma$ for corresponding $i, j$.
\end{lemma}
\begin{proof}
    Projection reduces the number of adjacent cells with different spins. Also, the energy of each cell decreases along the given projection, so we have the desired inequalities.
\end{proof}

Now we prove Theorem \ref{thm:gamma}.
\begin{proof}[Proof of Theorem \ref{thm:gamma}]
    Consider a path $\gamma$ from \textbf{1} to \textbf{3}. The projected path $\mc P_{23} \gamma$ gives a path in the Ising model with spins of type 1 and 3, of which the communication height between \bf1 and \bf3 is $H(1) + \Gamma_{13}^\star$ as described in Appendix \ref{app:a1}. Together with Lemma \ref{lem:proj}, we have
    \[
        \sup_{\sigma \in \gamma} H(\sigma) \ge \sup_{\sigma \in \gamma} H(\mc P_{23}(\sigma)) \ge H(\bf1) + \Gamma_{13}^\star.
    \]
    Hence we obtain $\Gamma(\bf1, \bf3) \ge \Gamma_{13}^\star$. Projection by $\mathcal{P}_{13}$ similarly gives $\Gamma(\bf2, \bf3) \ge \Gamma_{23}^\star$. On the other hand, consider the \emph{reference path} of the Ising model described in Appendix \ref{app:a1}. Assumption A ensures that such path exists in our model. This gives an exact equality for the previous inequality. Thus, we obtain
    \[
        \Gamma(\bf1, \bf3) \le \Gamma_{13}^\star, \quad \Gamma(\bf2, \bf3) \le \Gamma_{23}^\star.
    \]

    Finally, it remains to compute $\Gamma(\bf1, \bf2)$. Projecting a path from \bf1 to \bf2 by $\mc P_{13}$ gives a path $\mc P_{13}\gamma$ from \bf1 to \bf3 in the Ising model with spins of type 1 and 3. By the same argument as above, we have $\Gamma(\mathbf{1}, \mathbf{2}) \ge \Gamma_{13}^\star$. Also, consider a path, denoted by $\gamma_0$, concatenating the reference path of the Ising model between \bf1 and \bf3 with the reference path of the Ising model between \bf3 and \bf2. The communication height of this path is the maximum of the communication height of two reference paths, which is
    \[
        \max(H(\bf1) + \Gamma_{13}^\star, H(\bf2) + \Gamma_{23}^\star) = H(\bf2) + \Gamma_{23}^\star
    \]
    following by the computation in the Lemma \ref{lem:app2}. Therefore, the equality $\Gamma(\mathbf{1}, \mathbf{2}) = \Gamma_{13}^\star$ holds.
\end{proof}

The same argument can be applied to compute $\mc V_1$ and $\mc V_2$.
\begin{proposition}
\label{prop:stab}
    The stability level of \bf 1 and \bf 2 is given by
    \[
        \mathcal{V}_{\bf1} = \Gamma^\star_{13},\quad \mathcal{V}_{\bf2} = \Gamma^\star_{23}.
    \]
\end{proposition}
\begin{proof}[Proof of Theorem \ref{thm:stab}]
    Let $\gamma$ be a path from \bf1 to $\sigma$ with $H(\sigma) < H(\bf1)$. Then the projection by $\mc P_{23}$ gives a path $\mc P_{23}\gamma$ from \bf1 to $\mc P_{23}(\sigma)$ in the Ising model with spins 1 and 3. Also, note that $H(\mc P_{23}(\sigma)) \le H(\sigma) < H(\bf 1)$ by Lemma \ref{lem:proj}. Hence, the energy of the projected path is bounded by the stability level of \bf1 in the Ising model, which is $\Gamma_{13}^\star$. Together, we have
    \[
        \sup_{\sigma \in \gamma} H(\sigma) - H(\bf1) \le \sup_{\sigma \in \gamma} H(\mc P_{23}(\sigma)) - H(\bf1) \le \Gamma_{13}^\star.
    \]
    Taking a minimum of over all possible $\gamma$, we have $\mathcal{V}_{\bf1} \ge \Gamma_{13}^\star$. On the other hand, the reference path of the Ising model between \bf1 and \bf3 gives the equality of the above inequality. Therefore, $\mc V_{\bf1} = \Gamma_{13}^\star$. Similarly, we may apply $\mathcal{P}_{13}$ to obtain $\mathcal{V}_{\bf2} = \Gamma_{23}^\star$.
\end{proof}

\subsection{Proof of stability level}
\label{subsec:stab}
In this subsection, we give a proof of Theorem \ref{thm:stab2}. We clarify that the entire subsection is designed to include the proof of Theorem \ref{thm:stab2}, without declaring the beginning and the end of the proof. We use this structure due to a number of lemmas and cases involved in the proof.

From the previous proposition and Appendix \ref{app:a2}, $\mc V_{\bf 1} < \mc V_{\bf 2}$ immediately follows. Thus, it remains to prove $\mc V_{\sigma} < \mc V_{\bf1}$ for $\sigma \in \XX \backslash \{\bf1, \bf2, \bf3\}$. A simple computation in the Appendix \ref{app:a2} shows that $\mc V_{\bf 1} = \Gamma_{13}^\star > 2$. Thus, it suffices to check the case where $\mathcal{V}_\sigma > 2$.
\begin{figure}
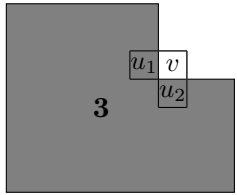

\centering
\ctikzfig{rect}
\caption{\label{fig:rect} Figure for Lemma \ref{lem:rect}. Flipping the spin of $v$ to 3 decreases the energy}
\end{figure}

\begin{lemma}
    \label{lem:rect}
    Suppose that $\mathcal{V}_\sigma > 0$. Then each 3-cluster, a connected component of 3-spins, of $\sigma$ forms a rectangle. Moreover, each rectangle has a distance of at least 2 from the others.
\end{lemma}
\begin{proof}
    If a cell has at least two neighborhoods with 3-spins, flipping its spin to 3 decreases the energy. Consider a 3-cluster $D$ of $\sigma$. Suppose $D$ has internal angle $\frac{3}{2}\pi$ as in Figure \ref{fig:rect} where $u_1, u_2 \in C$, $v \notin C$. Then flipping the spin of $v$ to 3 decreases the energy, so $\mc V_\sigma = 0$. Therefore $D$ must have internal angles $\pi$ or $\frac{1}{2}\pi$, which is a rectangle.

    Similarly, if two rectangles of 3-spins are spaced 2 units apart, we can find a cell with a spin other than 3 that is adjacent to both rectangles. Changing its spin into 3 decreases the energy, so $\mc V_\sigma = 0$, which is a contradiction.
\end{proof}

Now we divide into two cases.

\vspace{5pt}
\noindent
\bf{(Case 1: $\sigma$ has no 3-cluster)} Assuming $\mathcal{V}_\sigma > 0$, we may imitate the proof of Lemma \ref{lem:rect} to show that 2-clusters form a rectangle. Suppose $\sigma$ has both 1 and 2 spins, let $D'$ be a 2-cluster of size $l \times m$.

If $l < \frac{2}{h_2 - h_1}$, we may consecutively flip 2-spins into 1-spins along the length $l$ side of $D'$. This gives a 2-cluster of size $l \times (m - 1)$. Each flip increases the energy by $h_2 - h_1$ except for the last flip which decreases the energy by $2 - (h_2 - h_1)$, and the energy difference in total is $l(h_2 - h_1) - 2 < 0$. So we have $\mathcal{V}_\sigma < (l-1)(h_2 - h_1) < 2$.

If $l > \frac{2}{h_2 - h_1}$, we flip 1-spins into 2-spins along the length $l$ side of $D'$ to obtain a 2-cluster of size $l \times (m + 1)$. The first flip increases the energy by $2 - (h_2 - h_1)$, then it decreases by $h_2 - h_1$. The final energy difference is $2 - l(h_2 - h_1) < 0$, so we have $\mathcal{V}_\sigma < 2$.

\vspace{5pt}
\noindent \bf{(Case 2: $\sigma$ has 3-cluster)} Let $D$ be a 3-cluster. By Lemma \ref{lem:rect}, we may assume $D$ is a rectangle of size $l \times m$.

\begin{figure}
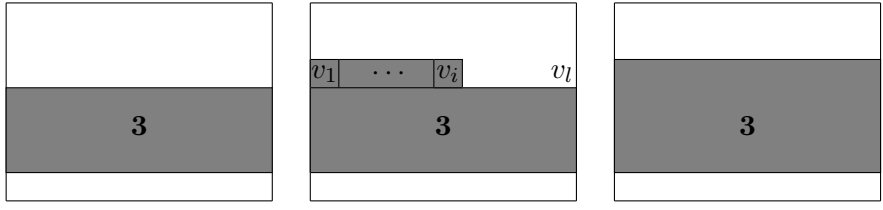

\centering
\ctikzfig{strip}
\caption{\label{fig:strip} Energy decreasing path for a configuration with two-sided 3-cluster. Flipping the cells $v_1, \ldots, v_l$ in order decreases the energy}
\end{figure}

First, we consider the case where $D$ is two-sided as in Figure \ref{fig:strip}. Let $v_1, \ldots, v_l$ be the cells adjacent to one side of $D$. Define the path $\omega = (\omega_0, \ldots, \omega_l)$ by flipping the spins of $v_i$ consecutively to 3. As a result, the width of $D$ increases by one. Denote the spin of $v_i$ in $\sigma$ as $\sigma(v_i)$. Then a simple computation shows that
\[
    H(\omega_i) - H(\omega_0) \le 2, \quad H(\omega_l) - H(\omega_0) = \sum_{i=1}^l (h_{\sigma(v_i)} - h_3) \le 0.
\]
If $H(\omega_l) = H(\omega_0)$, we may repeat the same action which increases the width of $D$. This eventually decreases the energy. Hence $\mathcal{V}_\sigma \le 2$.

Now we may assume that $D$ has four sides. We focus on the neighborhood cells of $D$ including the corners, which are $2(l+m+2)$ cells in total. By Lemma \ref{lem:rect}, these cells have spin 1 or 2. Among $2(l+m+2)$ edges lying between these cells, we focus on the edge where the spin changes. We call them \emph{colored} edges. Formally, we say that the edge is \emph{colored} if two cells containing the edge have different spins.

\begin{figure}
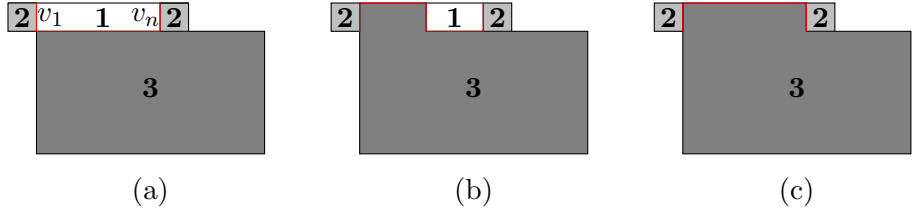

\centering
\ctikzfig{col1}
\caption{\label{fig:col1}Illustration regarding the proof of Lemma \ref{lem:col1}.}
\end{figure}

We claim that Figure \ref{fig:col1}-(a) has a stability level at most 1. Denote by $D$ the rectangular 3-cluster in the configuration.
\begin{lemma}
    \label{lem:col1}
    If two or more colored edge exist on one side of $D$, then $\mathcal{V}_\sigma \le 1$.
\end{lemma}
\begin{proof}
    Let $v_1, \ldots, v_n$ be the maximal consecutive cells on one side of $D$ with the same spin, as in Figure \ref{fig:col1}. By assumption, neither $v_1$ nor $v_n$ is a corner of $D$. For simplicity, we assume that the spins of $v_i$'s are 1. It will turn out that the proof is independent of the spin type.

    Define the path $\omega = (\omega_0, \ldots, \omega_n)$ by flipping the spins of $v_i$ consecutively to 3. To compute $H(\omega_i) - H(\omega_0)$, it suffices to compare $v_1, \ldots, v_i$. Counting the number of the colored edges of $v_1, \ldots, v_i$, we see that $\omega_0$ contains at least $i+1$ colored edges, and $\omega_i$ contains at most $i+2$ colored edges for $0 < i < n$. The red edges in Figure \ref{fig:col1}-(a), (b) show the possible colored edges. Thus, we have
    \[
        H(\omega_i) - H(\omega_0) \le 1 - i(h_3 - h_1) \le 1.
    \]
    If $i = n$, $\omega_0$ contains at least $n+2$ colored edges among the edges of $v_1, \ldots, v_n$, and $\omega_n$ contains at most $n+2$ colored edges, see the red edges in Figure \ref{fig:col1}-(a), (c). Thus, we have
    \[
        H(\omega_n) - H(\omega_0) \le -n(h_3 - h_1) < 0.
    \]
    Therefore, we have $\mc V_\sigma \le 1$.
\end{proof}

\begin{figure}
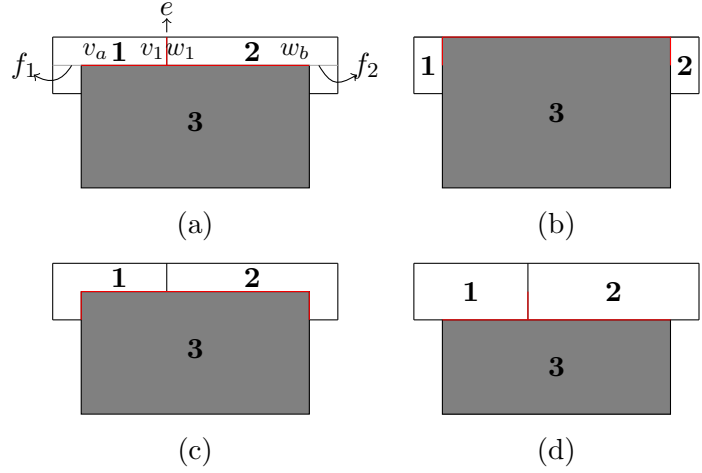

\centering
\ctikzfig{col2}
\caption{\label{fig:col2} Figure for Lemma \ref{lem:col2} where $a, b \neq 0$. The red edges indicate the possible boundaries between different spins. The top two figures show the process of flipping the spins of $v_1, \ldots, v_a$ and $w_1, \ldots w_b$ successively. The bottom two figures are obtained through a similar process. Either one of these processes decreases the total energy of the configuration.}
\end{figure}

\begin{figure}[!ht]
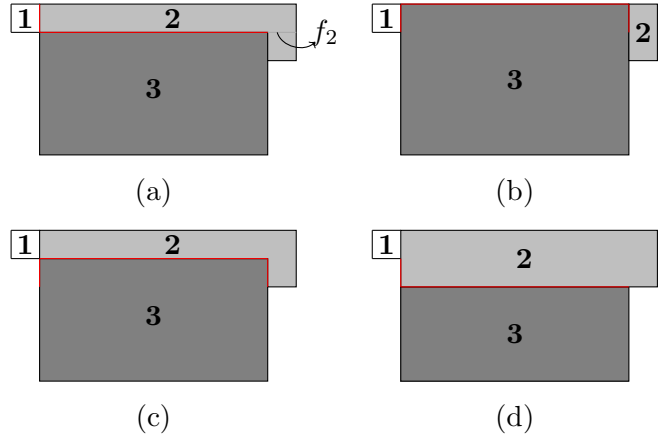

\centering
\ctikzfig{col3}
\caption{\label{fig:col3} Figure for Lemma \ref{lem:col2} where $a = 0$. The right two figures show the resulting configurations by flipping the spins along the rows. If $f_2$ is not colored, either of these processes decreases the total energy of the configuration.}
\end{figure}

The next lemma shows that a colored edge on one side forces the existence of a colored edge on the other side.
\begin{lemma}
    \label{lem:col2}
    Suppose that $\mc V_\sigma > 1$. Along one side of the 3-cluster $D$, suppose there is a unique colored edge $e$ splitting the side of $D$ with length $a$ and $b$. Consider the cells around $D$, which share an edge or a vertex with $D$. Starting from the edge $e$, denote these cells in clockwise order by $w_1, \ldots, w_{b+2}$. Also, starting from $e$, denote the cells by $v_1, \ldots ,v_{a+2}$ in counterclockwise order. Let $f_1$ be the edge between $v_{a+1}$ and $v_{a+2}$, and $f_2$ be the edge between $w_{b+1}$ and $w_{b+2}$. Refer to Figure \ref{fig:col2}-(a) for the precise illustrations.
    
    Suppose $a, b > 0$. If $e$ is colored, then $f_1$ or $f_2$ is colored. Moreover, if $a = 0$, then $f_2$ is colored.
\end{lemma}
\begin{proof}
    Suppose $e$ is colored but $f_1, f_2$ are not colored. By Lemma \ref{lem:col1} and the assumption $\mc V_\sigma > 2$, the cells $v_1, \ldots, v_{a+2}$ have same spins and $w_1, \ldots, w_{b+2}$ have same spins. Without loss of generality, let $v_i$'s have spin 1 and $w_j$'s have spin 2.

    Let $\sigma_1$ be a configuration obtained by flipping the spins of $v_1, \ldots, v_a$ and $w_1, \ldots, w_b$ to 3, as in Figure \ref{fig:col2}-(b). Let $\sigma_2$ be a configuration obtained by flipping the spins of cells below $v_1, \ldots, v_a$ to 1, and the cells below $w_1, \ldots, w_b$ to 2 as in Figure \ref{fig:col2}-(d). We compare the energy of these two states with $\sigma$.
    
    Figure \ref{fig:col2}-(a),(b) compare $\sigma_1$ and $\sigma$. The figure shows the following observations.
    \begin{itemize}[leftmargin=*]
        \item Among the edges of the flipped cells, $\sigma$ has at least $l+1$ colored edges. 
        \item Among the edges of the flipped cells, $\sigma_1$ has exactly $l + 2$ colored edges. This is because the distances between 3-clusters are at least 2, as stated in Lemma \ref{lem:rect}.
    \end{itemize}
    As before, possible colored edges are colored red in Figure \ref{fig:col2}. Hence we have
    \[
        H(\sigma_1) - H(\sigma) \le -(h_3 - h_1)a - (h_3 - h_2)b + 1.
    \]

    Figure \ref{fig:col2}-(c),(d) compare $\sigma_2$ and $\sigma$. The figure shows the following observations.
    \begin{itemize}[leftmargin=*]
        \item Among the edges of the flipped cells, $\sigma$ has exactly $l + 2$ colored edges.
        \item Among the edges of the flipped cells, $\sigma_2$ has at most $l + 1$ colored edges, due to the case where the height of $D$ is 1.
    \end{itemize}
    Hence we have
    \[
        H(\sigma_2) - H(\sigma) \le (h_3 - h_1)a + (h_3 - h_2)b - 1.
    \]
    Recall from Assumption B that $(h_3-h_1)a + (h_3 - h_2)b - 1$ cannot be zero. Hence, we have either
    $H(\sigma_1) < H(\sigma)$ or $H(\sigma_2) < H(\sigma)$.

    Finally, we construct a path from $\sigma$ to $\sigma_1$ and $\sigma_2$ by flipping the spins of consecutive cells. By bounding the energy of the path by $H(\sigma)+1$, one can show that $\mc V_{\sigma} \le 1$ as in the proof of Lemma \ref{lem:col1}. First, suppose $H(\sigma_1) < H(\sigma)$. Consider a path from $\sigma$ to $\sigma_1$ by flipping $v_1, \ldots, v_a$, then $w_1, \ldots ,w_b$ in order. The first flip increases the energy by $1 - (h_3 - h_1)$, then the following steps decrease the energy. So the energy of the path is bounded by $H(\sigma) + 1$ and we have $V_\sigma \le 1$. On the other hand, suppose $H(\sigma_2) < H(\sigma)$. Construct a path from $\sigma$ to $\sigma_2$ by sequentially flipping cells from the outermost to the innermost. Begin by flipping the 2-spins. At each step, when a 1-spin or 2-spin is transformed into a 3-spin, the energy increases by at most $h_3 - h_1$ or $h_3 - h_2$, respectively. Meanwhile, the last flip decreases the energy by $1 - (h_3 - h_2)$. Thus, the energy of the path is bounded by $H(\sigma_2) + 1$, which is less than $H(\sigma) + 1$.
    
    Therefore, we have a path from $\sigma$ to a state with lower energy, where the energy of the path is bounded by $H(\sigma) + 1$. This contradicts the assumption $\mc V_\sigma > 1$, so $f_1$ or $f_2$ must be colored.

    We may handle the case $a = 0$ similarly. Suppose that $f_2$ is not colored. Note that $f_1$ may be colored here. With the same notation as above, Figure \ref{fig:col3} shows that
    \[
        H(\sigma_1) - H(\sigma) \le (h_3 - h_2)l - 1,\quad H(\sigma_2) - H(\sigma) \le -(h_3 - h_2)l + 1.
    \]
    Either $\sigma_1$ and $\sigma_2$ have energy less than $H(\sigma)$, and both have a path from $\sigma$ obtained by flipping consecutive cells as before. This path has energy less than $H(\sigma) + 1$, so $V_\sigma < 1$, which is a contradiction. Therefore $f_2$ must be colored.
\end{proof}
The previous lemma provides a restriction on the possible position of the colored edge along the boundary of $D$. Using the notations from Lemma \ref{lem:col2} and Figure \ref{fig:col2}-(a), we claim that there is no colored edge $e$ with $a$ and $b$ both greater than 0.

Suppose to the contrary that such $e$ exists, which we will denote as $e_1$. Now, we reuse the notations of Figure \ref{fig:col2}-(a) where $e_{i+1}$ takes on the role of $e$ and $a = 0$. Define the edge corresponding to $f_2$ as $e_{i+1}$ for $i=1, 2, 3, 4$, as shown in Figure \ref{fig:windmill-mid}. By the previous lemma, $e_2$ is colored. Then inductively, $e_3$, $e_4$, $e_5$ must be colored. Since $e_1$ and $e_5$ is colored, this contradicts $V_\sigma > 2$ by Lemma \ref{lem:col1}.

Hence we may assume that there is no colored edge at the middle of the side of $D$. The remaining case is when the 3-cluster is surrounded by a single type of spin, or when the neighborhood spins are located as Figure \ref{fig:windmill}-(a).

\begin{figure}
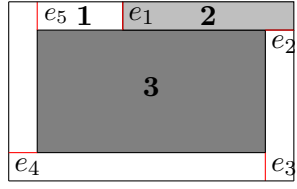

\centering
\ctikzfig{windmill-mid}
\caption{\label{fig:windmill-mid} The figure shows that two spins cannot lie on the same side. If $e_1$ is colored, so are $e_2, \ldots, e_5$ inductively by Lemma \ref{lem:col2}.}
\end{figure}

\vspace{5pt}
\noindent \bf{(Case 2-1: 3-cluster is surrounded by a single type of spin)} Let the surrounding spin type be $i$. We flip 3-spin at the down left corner to $i$-spin. Then we flip $i$-spin at the opposite side to 3-spin. Repeat this along the left side. As a result, the 3-cluster is shifted to the right by one. Along this operation, the energy is bounded above by $H(\sigma) + 2$. Moreover, the final configuration has the same number of spin as before, and the number of colored edges remains the same or decreases. Repeating this operation, the energy eventually decreases if the 3-cluster meets another cluster. So we have $V_\sigma < 2$ in this case. If the 3-cluster is the only cluster in the configuration, we may apply the same argument as in Case 1. Therefore $V_\sigma < 2$.

\begin{figure}
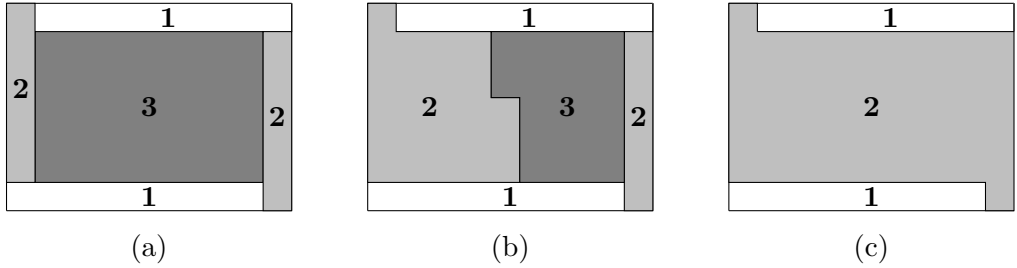

\centering
\ctikzfig{windmill}
\caption{\label{fig:windmill} Figure for Case 2-2 in the proof of Theorem \ref{thm:stab2}. Starting from (a), 2-spins are flipped to 3-spins from left to right column, downwards in each column. The resulting configuration (c) has lower energy than (a).}
\end{figure}

\vspace{5pt}
\noindent \bf{(Case 2-2: 3-cluster is surrounded by two type of spins)} Lemma \ref{lem:col2} asserts that Figure \ref{fig:windmill}-(a) is the only possible configuration. Consider the following two operations.
\begin{enumerate}[leftmargin=*]
    \item Successively flip the 1-spins along the side of length $a$ to 3-spins. This changes the energy no more than 2. The final configuration $\sigma_a$ has at most $a + 2$ colored edge around the flipped cells. The initial state has at least $a+1$ colored edge, so we have
    \[
        H(\sigma_a) - H(\sigma) \le -(h_3-h_1)a + 1.
    \]
    If this bound is below zero, we have $V_\sigma \le 2$.
    \item Do the same operation on the side of length $b$. Then we have either $V_\sigma \le 2$ or $b \le \frac{1}{h_3-h_2}$.
\end{enumerate}
Hence we may assume $a \le \frac{1}{h_3-h_1}$ and $b \le \frac{1}{h_3-h_2}$. Now, consider the path shown in Figure \ref{fig:windmill}.
Flip the 3-spins of the rectangle into 2-spins from the left column to the right, each column from the bottom to the top. Denote this path by $\{\gamma_n\}_{n=0,\ldots,ab}$. Let $0 \le q \le a - 1$ and $0 < r \le b$ be integers such that $n = bq + r$. Then we have
\[
    H(\gamma_n) - H(\sigma) = \begin{cases}
        (h_3 - h_2)(bq + r) + \one_{\{r \neq 0\}} & q < a - 1 \\
        (h_3 - h_2)(bq + r) - (2r - 1) & q = a - 1,\quad r < b \\
        (h_3 - h_2)ab - 2b  & q = a - 1,\quad r = b.
    \end{cases}
\]
We have
\[
    H(\gamma_{ab}) - H(\sigma) = \left((h_3 - h_2)a - 2\right)b < 0, \text{ and}
\]
\[
    \max_{n = 0, \ldots, ab} H(\gamma_n) - H(\sigma) = (h_3 - h_2)(a-1)b < (h_3 - h_2) \cdot \frac{1}{(h_3-h_2)}\cdot \frac{1}{(h_3-h_1)} < V_1.
\]
Therefore, $V_\sigma < V_1$. This completes the proof of Theorem \ref{thm:stab2}.

Finally, we determine the minimal gates between \bf1, \bf2, and \bf3.
\begin{proof}[Proof of Theorem \ref{thm:gate}]
    Recall the proof of Theorem \ref{thm:gamma}. For a path $\gamma$ from \bf 1 to \bf 3, we observed that $\mc P_{23}\gamma$ gives a path in the Ising model with spins of type 1 and 3, and thus,
    \[
        \sup_{\sigma \in \gamma} H(\sigma) \ge \sup_{\sigma \in \gamma} H(\mc P_{23}(\sigma)) \ge H(\bf1) + \Gamma_{13}^\star.
    \]
    For the optimal path $\gamma \in \Omega_{\bf1, \bf3}$, each equality of the above inequalities must hold. Thus, $\mc P_{23}\gamma$ must be an optimal path of the Ising model. So $\mc P_{23}(\eta) \in \mathcal{G}_{13}$ for some $\eta \in \gamma$. Then, we have
    \[
        \Phi(\bf1,\bf3) = \sup_{\sigma \in \gamma} H(\mc P_{23} \sigma) = H(\mc P_{23} \eta) \le H(\eta) \le \sup_{\sigma \in \gamma}H(\sigma) = \Phi(\bf1, \bf3),
    \]
    so $H(\eta) = H(\mc P_{13}(\eta))$ are both equal to $\Phi(\bf1,\bf3)$. By Lemma \ref{lem:proj}, this is the case only if $\eta = \mc P_{13}(\eta)$. Therefore $\gamma$ meets $\mc G_{13}$.

    Now we consider the reference path of the Ising model between \bf1 and \bf3. Note that each configuration in $\mc G_{13} = \dect 3 1 {\ell_c^{13}+1} {\ell_c^{13}}$ is contained in some reference path, denoted by $\gamma_{13}$. Since $\mc P_{23}(\gamma_{13}) = \gamma_{13}$, $\gamma_{13}$ is an optimal path in $\Omega_{\bf1,\bf3}$ in our model. Relying on the knowledge on the reference path provided in the Appendix \ref{app:a1}, $\gamma_{13}$ has a unique saddle configuration. Namely, $|\gamma_{13} \cap \mc G_{13}| = 1$ and
    \[
        \sup_{\sigma \in \gamma_{13}, \sigma \notin \mc G_{13}} H(\sigma) < \Phi(\bf1, \bf3).
    \]
    This tells that each configuration in $\mc G_{13}$ is \emph{essential}, which means that it must be contained in any gate between \bf1 and \bf3. Hence, $\mc G_{13}$ is a minimal gate between \bf1 and \bf3.

    Similarly, we may apply $\mc P_{13}$ and $\mc P_{23}$ to show the remaining statements.
\end{proof}

\subsection{Large deviation-type results}
\label{subsec:path_proof}
We prove Theorem \ref{thm:LDP} and Proposition \ref{prop:pot_zero} using
a pathwise approach described in Section \ref{subsec:pathwise}. To achieve this, we identify the tube of typical paths and show that \bf2 does not appear in such paths. It is worth noting that in other literature, such precise study is not needed since transition time, a usual interest, can be estimated by a rough knowledge such as minimal gates. However, we need that \bf2 is not contained in all typical paths, which requires a more precise analysis.

We start by classifying the cycles along the typical path.
\begin{lemma}
    \label{lem:cycle}
    Recall the definition of $\ell_c^{13}$ from \eqref{eq:crit}. Let $i, j$ be at least $\ell_c^{13} = \lceil 2/(h_3-h_1) \rceil$ and $\eta_0$ be a single configuration $\rect 3 1 i j $. Let $C$ be a set of configurations characterized by the following properties, depicted in Figure \ref{fig:cycle}.
    \begin{enumerate}[(1), leftmargin=*]
        \item The configuration does not contain 2-spins.
        \item There is a unique 3-cluster, and the contour of the 3-cluster is \emph{simple}, in the sense that every intersection between parallel or vertical lines within the interior of the 3-cluster is connected.
        \item The unique 3-cluster of the configuration is contained in the rectangular 3-cluster of $\eta_0$.
        \item The number of 3-spins is at least $ij - \ell_c^{13} + 2$.
    \end{enumerate}
    Then $C$ is a maximal cycle disjoint from \bf3 and containing $\eta_0$. In particular, we have the following properties on $C$.
    \begin{enumerate}[(a), leftmargin=*]
        \item The bottom of $C$ is $\eta_0$. The depth of $C$ is $2 - (h_3 - h_1)$.
        \item The bottom of $\partial C$ lies in $\dect 3 1 i j \cup \dect 3 1 j i$.
    \end{enumerate}
\end{lemma}
\begin{figure}
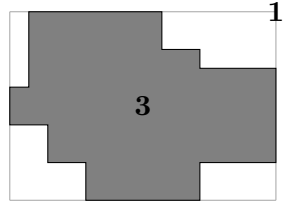

\centering
\ctikzfig{cycle}
\caption{\label{fig:cycle} A 3-cluster and the surrounding rectangle. Lemma \ref{lem:cycle}-(2) claims that for the configurations in the cycle $C$, the surrounding rectangle remains fixed.}
\end{figure}
\begin{proof}
    First, we show that $C$ is a cycle. It suffices to show the following claims.
    \begin{claim}
        For $\sigma \in C$, we have $H(\eta_0) \le H(\sigma) < H(\eta_0) + 2 - (h_3 - h_1)$.
    \end{claim}
    \begin{proof}[Proof of Claim]
    Observe that the perimeter of the 3-cluster of $\sigma$ is $2(i+j)$, which is same as that of $\eta_0$. So counting the number of 1-spins yields
    \begin{align*}
        0 \le H(\sigma) - H(\eta_0) = (h_3 - h_1)\left(ij - \sum_{x \in V} \one_{\{\sigma(x) = 1\}} \right) \le (h_3-h_1)(\ell_c - 2) < 2 - (h_3 - h_1)
    \end{align*}
    where $V$ denotes the set of cells of the configuration. Therefore the claim holds.
    \end{proof}

    \begin{claim}
        For $\sigma \in \partial C$, we have $H(\sigma) \ge H(\eta_0) + 2 - (h_3 - h_1)$.
    \end{claim}
    \begin{proof}[Proof of Claim]
    Let $\sigma'$ be a configuration in C adjacent to $\sigma$. We show by considering the possible shapes of $\sigma$. Suppose $\sigma$ is obtained from $\sigma$ by changing a $i$-spin to a 2-spin. Then the number of the pairs of adjacent distinct spins increases by at least 2, so
    \[
        H(\sigma') - H(\sigma) \ge 2 + (h_i - h_2) > 2 - (h_3 - h_1).
    \]
    
    Next, suppose $\sigma$ is obtained from $\sigma$ by changing a 3-spin to a 1-spin. Note that the perimeter of the 3-cluster cannot decrease. If this is the case, we need to remove at least $\min(i, j) \ge \ell_c^{13}$ 3-spins compared to $\eta_0$. This contradicts the property (4) of the set $C$. 
    If the perimeter of the 3-cluster increases, it increases by at least 2. So the energy increases by at least $2 + (h_3 - h_1)$, so we are done. If the perimeter is preserved, we see that $\sigma'$ must be a shape of Figure \ref{fig:cycle}, or explicitly, satisfy (1), (2), and (3) of the above properties. Hence $\sigma'$ consists of exactly $ij - \ell_c^{13} + 1$ 3-spins, and we have
    \[
        H(\sigma') - H(\eta_0) = (h_3-h_1)(\ell_c^{13}+1) > 2 - (h_3 - h_1).
    \]

    Finally, suppose $\sigma'$ is obtained from $\sigma$ by replacing a 1-spin by 3-spin. Then the parameter of the 3-cluster is either preserved or increases. If the perimeter increases, the energy difference is bounded by
    \begin{equation}
        \label{eq:cyc_bdry}
        H(\sigma') - H(\eta_0) \ge H(\sigma') - H(\sigma) \ge 2 - (h_3 - h_1),
    \end{equation}
    so the claim holds. If the perimeter is preserved, we again see that $\sigma'$ must be as in Figure \ref{fig:cycle}, satisfying properties (1), (2), and (3). Moreover, the number of 3-spin decreases, so (4) is satisfied. So $\sigma' \in C$, which contradicts $\sigma' \in \partial C$.
    \end{proof}
    
    From the above two claims, we have that $C$ is a cycle. Moreover, the claims directly imply (a) of the statement. To see (b), it suffices to see when the equality holds while showing $H(\sigma) \le H(\eta_0) + 2 - (h_3 - h_1)$ for $\sigma \in \partial C$ in the claim. Following the proof, \eqref{eq:cyc_bdry} is the only possible case. This occurs when $\sigma'$ is obtained from $\sigma$ by changing a 1-spin attached to a boundary of 3-cluster into 3-spin, and moreover $H(\sigma) = H(\eta_0)$. Since $\eta_0$ is the unique bottom of $C$, we have $\sigma = \eta_0$. Therefore, $\sigma'$ lies in $\dect 3 1 i j \cup \dect 3 1 j i$.

    Finally, we show that $C$ is a maximal cycle containing $\eta_0$ and disjoint from \bf3. Showing that $\Gamma(\eta_0, \bf3) = 2 - (h_3 - h_1)$, which is a depth of $C$, gives the result. Indeed, a cycle containing $C$ has a larger depth, so it must contain \bf3. The cycle $C$ itself gives a lower bound $\Gamma(\eta_0, \bf3) \ge 2 - (h_3 - h_1)$, so it suffices to construct an optimal path $\gamma$ of height $H(\eta_0) + 2 - (h_3 - h_1)$. The construction is similar to the reference path of the Ising model. Precisely, starting from $\eta_0$, add a single 1-spin to the edge of the 3-cluster. Then, flip the rest of the 1-spins on the edge to 3-spins, which strictly decreases the energy. This operation gives a 3-cluster of size $(i + 1) \times j$, and decreases the energy by $2 - j(h_3 - h_1) < 0$ in total. Repeating the same operation, we may construct a path with the maximal energy less than $\Phi(\eta_0, \bf3)$. Hence, $C$ is the maximal cycle.
\end{proof}

We remark that in the proof, we did not use the relation $h_1 < h_2$. Thus a similar statement holds between spins of type 1 and 2, where $\ell_c^{13}$ is replaced by $\ell_c^{12}$ and the depth of the cycle is replaced by $2 - (h_3 - h_2)$.

Now we can fully describe typical paths from \bf1 to \bf3. This will lead us to the proof of Proposition \ref{prop:pot_zero}.
\begin{proof}[Proof of Proposition \ref{prop:pot_zero}]
    Consider a state in $\dect 3 1 i j$ where $i$ and $j$ are both greater than $\ell_c^{13}$. In this state, the only way to decrease the energy is by filling the 3-spins on the length $i$ side of the 3-cluster or returning to $\rect 3 1 i j$. Therefore, we move from $\rect 3 1 i j$ to either $\rect 3 1 {i+1} j$ or $\rect 3 1 i {j+1}$ with a probability of $1 - o_\beta(1)$. This follows from Proposition \ref{prop:exit} and the previous lemma.

    Now we determine a typical path from \bf1 to \bf3. The process exits the cycle of \bf1 through $\mc G_{13} = \dect 3 1 {\ell_c^{13}+1}{\ell_c^{13}}$ with a probability of $1 - o_\beta(1)$. Then, it follows a downhill path towards $\rect 3 1 {\ell_c^{13}+1}{\ell_c^{13}}$. This sequence of steps is repeated until we reach \bf3. This provides a comprehensive description of a typical path. Let $P$ be a union of all typical paths from \bf1 to \bf3. Note that $P$ does not contain any states with 2-spins. By Lemma \ref{lem:tube}, we obtain
    \[
        \limb \PP_{\bf1}^\beta[\tau_{\bf2} < \tau_{\bf3}] \le \limb \PP_{\bf1}^\beta[\tau_{\XX \backslash P} < \tau_{\bf3}] = 0.
    \]
\end{proof}

The main difficulty of the proof was that the tube of typical trajectories from \bf1 to \bf3 avoids \bf2. With this fact, one can also derive a large deviation-type result on a transition from \bf1 to \bf3.
\begin{proof}[Proof of Theorem \ref{thm:LDP}]
    A result on $\bf2 \to \bf3$ transition readily follows since \bf2 is the state with maximal stability. Indeed, \cite[Theorems 3.17 and 3.19, Propositions 3.18 and 3.20]{Nardi_2015} directly gives
        \[
        \lim_{\beta\to\infty} \frac{1}{\beta}\log\EE_{\mathbf{2}}^\beta[\tau_{\mathbf{3}}] = \Gamma_{23}^\star
    \]
    and
    \[
        \frac{\tau_{\bf3}}{\EE_{\bf2}^\beta[\tau_{\bf3}]} \xrightarrow{d} \exp(1).
    \]

    To show the remaining result, we may apply \cite[Corollary 3.16]{Nardi_2015}. It states that if a typical tube $T_A(\sigma)$ avoids a cycle deeper than $C_A(\sigma)$, then we may obtain an estimate
    \[
        \limb \PP_\sigma^\beta \left[ e^{\beta(\Theta(\sigma, A) - \ee)} < \tau_A < e^{\beta(\Theta(\sigma, A) + \ee)}\right] = 1
    \]
    for all $\ee > 0$ where $\Theta(\sigma, A)$ is the depth of the cycle $C_A(\sigma)$. The transition $\bf2 \to \bf3$ clearly satisfies this condition since $\bf2$ has the maximal stability. In the case of transition $\bf1 \to \bf3$, we saw in the previous proof that the tube of typical paths avoids \bf2. Since \bf1 is the state with the second greatest stability level, as proved in Theorem \ref{thm:stab2}, 
    transition $\bf1 \to \bf3$ also satisfies the required condition. Hence, we have the desired results.
\end{proof}

\section{Potential Theory and Proofs}
\subsection{Notations and preliminaries}
The potential theory provides a way to compute the prefactor of the mean hitting time. We recall several concepts from potential theory. Given a function $f: \XX \to \RR$, the \emph{Dirichlet form} $\mathcal{D}_\beta(f)$ is defined as
\[
    \mathcal{D}_\beta(f) = \frac{1}{2} \sum_{\sigma, \eta \in \XX} \mu_\beta(\sigma)c_\beta(\sigma, \eta) \left[f(\sigma) - f(\eta)\right]^2.
\]
For disjoint subsets $A, B \subseteq \XX$, define the \emph{capacity} between $A$ and $B$ as
\[
    \capa(A, B) = \inf_{h:\XX \to \RR} \mathcal{D}_\beta(h)
\]
where the infimum runs over a function $h$ satisfying $h|_A = 1$ and $h|_B = 0$. The infimum attains a unique minimum at the \emph{potential function} $h_{A, B}$, which is defined as
\[
    h_{A, B}(\sigma) = \PP_{\sigma}^\beta[\tau_A < \tau_B]
\]
for $\sigma \notin A \cup B$.

Recall from \cite[Proposition 6.10]{Beltr_n_2010} the following proposition.
\begin{proposition}
    \label{prop:EKgen}
    Let $A$ and $B$ be two disjoint subsets of $\XX$. For a $\mu$-integrable function $g:E \to \RR$, we have
    \[
        \EE_{\nu_{AB}}^\beta \left[\int_0^{\tau_B}g(X_t)dt\right] = \frac{\inn{g, h_{A,B}}_\mu}{\capa(A, B)}.
    \]
    Here, $\nu_{AB}$ is the equibilium measure on $A$ defined as
    \[
        \nu_{AB}(\sigma) = \frac{\mu_\beta(\sigma)\PP_\sigma^\beta[\tau_B^+ < \tau_A^+]}{\capa(A,B)}.
    \]
\end{proposition}
By putting $A = \{\eta\}$, $B = \{\eta'\}$, and $g$ as a constant function, we obtain a formula on the mean hitting time.
\begin{proposition}
    \label{prop:EKcap}
    For any $\eta, \eta' \in \XX$, we have
    \[
        \EE_{\eta}^\beta[\tau_{\eta'}] = \frac{1}{\capa(\eta, \eta')} \sum_{\sigma \in \XX} \mu_\beta(\sigma)\PP_\sigma^\beta[\tau_{\eta} < \tau_{\eta'}].
    \]
\end{proposition}
The next two lemmas give an estimation of the terms appearing in the previous proposition. We refer to \cite[Lemma 8.4]{bovier2015metastability} for the first lemma, so-called the \emph{renewal estimate}.
\begin{lemma}
    \label{lem:pot_est} For two disjoint sets $A, B \subseteq \XX$, we have that
    \[
        \PP_\sigma^\beta[\tau_A < \tau_B] \le \frac{\capa(\sigma, A)}{\capa(\sigma, B)}.
    \]
\end{lemma}
We refer to \cite[Lemma 16.11]{bovier2015metastability} for the following lemma.
\begin{lemma}
    \label{lem:cap_est} For two disjoint sets $A, B \subseteq \XX$, there exists $0 < C_1 \le C_2 < \infty$ such that
    \[
        C_1 < e^{\beta\Phi(A,B)} Z_\beta \capa(A, B) < C_2
    \]
    for all $\beta \in (0, \infty)$.
\end{lemma}

\subsection{Capacity estimation}
\label{subsec:4.2}
In this section, we estimate the capacities between \bf1, \bf2, and \bf3. This is based on the work of \cite{Bovier_2002}, which computes the prefactor of the Ising model with an external field. However, due to the convenience of notations, we refer to the proof given in Section 16.3 of \cite{bovier2015metastability}. Although the statements found in \cite{bovier2015metastability} are between unique metastable and stable states, we may generalize them into arbitrary states by imposing some nice conditions. With the conditions below, the proof of the statements is merely the same.

Let $\xi$ and $\xi'$ be arbitrary distinct states in $\XX$, which appear to be metastable and stable states in the original works, and define $\Gamma^\star = \Phi(\xi, \xi') - H(\xi)$. Let $\mc G$ be a subset of $\XX$ with the following property.
\begin{enumerate}[(1), leftmargin=*]
    \item For all $\sigma \in \mc G$, $H(\sigma) = \Gamma^\star$.
    \item For all $\gamma \in \Omega_{\xi,\xi'}$, $\gamma \cap \mc G \neq \emptyset$.
    \item For all $\sigma \in \mc G$, there exists a path $\{\gamma_i\}_{i=1}^n \in \Omega_{\xi, \xi'}$ passing $\sigma$ such that $H(\gamma_i) < \Gamma^\star$ for all $\gamma_i \neq \sigma$.
\end{enumerate}
We may directly see from the definition that $\mc G$ is the unique minimal gate from $\xi$ to $\xi'$. Define the following subsets of $\XX$.
\begin{itemize}[leftmargin=*]
    \item $S^\star$ is obtained from $\XX$ by removing all vertices $\sigma$ with $H(\sigma) > \Gamma^\star$ and edges adjacent to these vertices, then taking the connected component of $m$ and $s$.
    \item $S^{\star\star}$ is obtained from $S^\star$ by removing $\mc G$.
\end{itemize}
Consider the connected components of $S^{\star\star}$. Note that by property (2) of $\mc G$, $\xi$ and $\xi'$ lie in different components. Let $S_\xi$ and $S_{\xi'}$ be the connected components containing $\xi$ and $\xi'$, respectively. Denote other components by $S_i$, $i = 1, \ldots, m$. Then we have
\begin{lemma} (\cite[Lemma 16.16]{bovier2015metastability})
    \label{lem:bov1}
    Recall the notation $h_{\xi,\xi'}(\sigma) = \PP_\sigma^\beta[\tau_\xi < \tau_{\xi'}]$.
    There exist constants $C < \infty$ and $\delta > 0$ such that for any two states $\sigma, \sigma'$ in a same component $S_i$,
    \[
        \abs{h_{\xi,\xi'}(\sigma) - h_{\xi,\xi'}(\sigma')} \le C e^{-\beta \delta}
    \]
    for all $\beta \in (0, \infty)$. 
\end{lemma}

\begin{lemma} (\cite[Lemma 16.17]{bovier2015metastability})
    \label{lem:bov2}
    The following estimate on the capacity holds as $\beta \to \infty$.
    \[
        Z_\beta \capa(\xi, \xi') = [1 + o_\beta(1)]\Theta e^{-\beta \Phi(\xi, \xi')}
    \]
    where
    \[
        \Theta = \min_h \frac{1}{2} \sum_{\sigma, \sigma' \in S^\star} \one_{\{\sigma \sim \sigma'\}} [h(\sigma) - h(\sigma')]^2
    \]
    subject to a function $h: S^\star \to [0, 1]$ satisfying $h|_{S_\xi} = 1$, $h|_{S_{\xi'}} = 0$, and $h|_{S_i}$ constant for each $i = 1, \ldots, m$.
\end{lemma}
The proof of two lemmas is exactly the same as in \cite{bovier2015metastability}. Lemma \ref{lem:bov1} follows from Lemma \ref{lem:cap_est} and Lemma \ref{lem:pot_est} and that $\Phi(\sigma, \sigma')$ is less than $\Phi(\xi, \xi')$. This implies that the potential function is nearly constant in each $S_i$. Now consider the equation
\[
    \capa(\xi, \xi') = \frac{1}{2Z_\beta} \sum_{\sigma, \sigma' \in \XX} e^{-\beta\Phi(\sigma, \sigma')}[h(\sigma) - h(\sigma')]^2,
\]
Lemma \ref{lem:bov1} implies that the potential function can be estimated by a constant in each component $S_i$. With this in mind, Lemma \ref{lem:bov2} follows by direct computation of Dirichlet form.

Finally, we obtain the capacity estimates between \textbf{1}, \textbf{2}, and \textbf{3} from the previous lemma.
\begin{proposition}
    \label{prop:cap}
    We have
    \begin{align*}
        \capa(\bf1, \bf3) = \kappa_1^{-1} e^{-\beta \Gamma^\star_{13}}(1 + o_\beta(1)) \\
        \capa(\bf2, \bf3) = \kappa_2^{-1} e^{-\beta \Gamma^\star_{23}}(1 + o_\beta(1)) \\
        \capa(\bf1, \bf2 \cup \bf3) = \kappa_1^{-1} e^{-\beta \Gamma^\star_{13}}(1 + o_\beta(1)) \\
        \capa(\bf2, \bf1 \cup \bf3) = \kappa_2^{-1} e^{-\beta \Gamma^\star_{23}}(1 + o_\beta(1)) \\
    \end{align*}
    where
    \begin{align*}
        \kappa_1 = \frac{3}{4(2\ell_c^{13} - 1)}\frac{1}{\abs{\XX}},\quad \kappa_2 = \frac{3}{4(2\ell_c^{23} - 1)}\frac{1}{\abs{\XX}}.
    \end{align*}
\end{proposition}

\begin{figure}
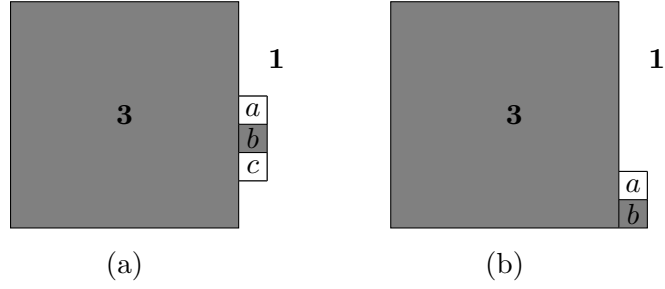

\centering
\ctikzfig{gate1}
\caption{\label{fig:gate} Two types of minimal gates between \bf1 and \bf3. Figure (a) has three adjacent states with lower energy. Figure (b) has two adjacent states with lower energy.}
\end{figure}

\begin{proof}
    For the first and third statements, assign $\mc G = \mc G_{13}$. Otherwise, assign $\mc G = \mc G_{23}$. Then $\mc G$ satisfies the above three conditions. This is followed by Theorem \ref{thm:gate}.

    To apply the previous lemma, we need to determine the connected components of $S^\star \backslash \mc G$. Each component must be adjacent to $\mc G$, so it suffices to examine adjacent states of $\mc G$ with energy no more than $\Gamma^\star$. There are two types of elements in $\mc G$, see Figure \ref{fig:gate}.
    \begin{enumerate}[label=\emph{Type \arabic*}:, leftmargin=*]
        \item Consider Figure \ref{fig:gate}-(a). If we add a different spin, the energy increases by at least $2 - (h_3 - h_1)$. Hence there are exactly three adjacent configurations lowering the energy: flipping the cells $a$, $b$, and $c$ into 3-spin. One lies in $S_m$, and two lie in $S_s$.
        \item Consider Figure \ref{fig:gate}-(b). Two adjacent configurations lower the energy, marked as $a, b$. One lies in $S_m$, and one lies in $S_s$.
    \end{enumerate}
    Therefore $S_m$ and $S_s$ are the only connected components of $S^\star \backslash \mc G$. Now we apply the variation principle to compute the capacity. We have
    \[
        \Theta = \min_{h:\mc G \to \RR} \frac 12 \Big{[}\sum_{\substack{\xi \in \mc G \\ \mc \xi: \text{ type 1}}} (1 - h(\xi))^2 + 2h(\xi)^2 + \sum_{\substack{\xi' \in \mc G \\ \mc \xi': \text{ type 2}}} (1 - h(\xi'))^2 + h(\xi')^2 \Big{]}.
    \]
    There are $2|\XX|$ rectangles of size $(\ell_c - 1) \times \ell_c$ or $(\ell_c - 1) \times \ell_c$, and there are 4 sites of type 1, $2\ell_c - 4$ sites of type 2. By Cauchy-Schwartz inequality, assigning $h \equiv \frac{1}{3}$ on type 1 states and $h \equiv \frac{1}{2}$ on type 2 states gives the minimum
    \[
        \Theta = \left[\frac{1}{2}\cdot 4 + \frac{2}{3}(2\ell_c - 4) \right] \cdot 2\abs{\XX} = \abs{\XX}\frac{4}{3}(2\ell_c - 1).
    \]
    Since $Z_\beta = e^{-\beta H(\bf3)}(1 + o_\beta(1))$, we have the desired results.
\end{proof}

Now, Theorem \ref{thm:EK} follows immediately.
\begin{proof}[Proof of Theorem \ref{thm:EK}]
    By Proposition \ref{prop:EKcap}
    \[
        \EE_2^\beta[\tau_3] = \frac{1}{\capa(\bf2, \bf3)} \sum_{\sigma \in \XX} \mu_\beta(\sigma)\PP_\sigma^\beta[\tau_2 < \tau_3].
    \]
    To bound the right-hand side, we show that the summation is dominated by $\mu_\beta(\bf2)$. For each $\sigma \neq \bf2, \bf3$, we have two cases:

    \vspace{5pt}
    \noindent
    \textbf{(Case 1: $H(\sigma) > H(\bf2)$)} Then 
    \[
        \limb \frac{1}{\mu_\beta(\bf2)} \mu_\beta(\sigma)h_{2,3}(\sigma) \le e^{-\beta(H(\sigma)-H(\bf2))} = 0,
    \]
    so $\mu_\beta(\sigma)h_{2,3}(\sigma)$ is dominated by $\mu_\beta(\bf2)$.

    \vspace{5pt}
    \noindent
    \textbf{(Case 2: $H(\sigma) \le H(\bf2)$)} This implies that $\sigma$ lies outside of the domain of attraction
    \[
        \{\eta: \Phi(\eta, \bf2) < \Phi(\bf2, \bf3)\},
    \]
    so $\Phi(\sigma, \bf2) \ge \Phi(\bf2, \bf3)$. Also, $\mc V_\sigma < \mc V_{23} = \Gamma^\star_{23}$. So we can move to a state with less energy, without visiting states with energy more than $H(\sigma) + \mc V_\sigma$. Since \bf2 has maximal stability, we may repeat this until we reach \bf3. Thus $\Phi(\sigma, \bf3) < H(\sigma) + \Gamma^\star_{23}$ and we have
    \begin{align*}
        \mu_\beta(\sigma)h_{2,3}(\sigma) & \le \mu_\beta(\bf2) e^{\beta(H(\bf2) - H(\sigma))} \frac{\capa(\sigma, \bf2)}{\capa(\sigma, \bf  3)} \\
        & \le \mu_\beta(\bf2) e^{\beta(H(\bf2) - H(\sigma))} e^{-\beta(\Phi(\sigma, \bf2) - \Phi(\sigma, \bf 3))} (1 + O_\beta(1))
    \end{align*}
    by Lemma \ref{lem:cap_est}. The power of the exponent is
    \[
        \beta\left(H(\bf2) - \Phi(\sigma, \bf2) - H(\sigma) + \Phi(\sigma, \bf3)\right) < \beta(H(\bf2) - \Phi(\bf2, \bf3) + \Gamma_{23}^\star) = 0.
    \]
    Hence we again have
    \[
        \limb \frac{1}{\mu_\beta(\bf2)} \mu_\beta(\sigma)h_{2,3}(\sigma) = 0.
    \]
    Gathering two cases gives an estimation
    \begin{align*}
        \EE_{\bf2}^\beta[\tau_{\bf3}]
        = \frac{1}{\capa(\bf2, \bf3)} \mu_\beta(\bf 2)(1 + o_\beta(1))
        = \kappa_2 e^{\beta \Gamma^\star_{23}}(1 + o_\beta(1)).
    \end{align*}
\end{proof}

\subsection{Markov chain model reduction}
\label{subsec:4.3}
In this section, we prove the model reduction stated in Section \ref{subsec:mcmr}. In \cite[Theorem 2.4]{Beltr_n_2010}, it states the sufficient condition to show model reduction: convergence of rate of the trace process, and negligibility of the time spent in non-metastable states. We show two conditions are satisfied in our model with respect to two different speed-up scales. Recall that two speed-up scales $\theta_\beta^1$ and $\theta_\beta^2$ are given by
\[
    \tb^1 = \kappa_1 e^{\beta\Gamma_{13}^\star},\quad \tb^2 = \kappa_2 e^{\beta\Gamma_{23}^\star}.
\]

Recall from \cite[Subsection 2.1]{Beltr_n_2010} the definition of \emph{trace process}. In brief, a trace process on $S \subseteq \XX$ is the original process but the clock stopped outside of $S$. This defines a continuous time Markov chain on $S$, and its jump rate can be computed in terms of capacities.

\begin{proposition}[Theorem \ref{thm:MCMR1}: Convergence of the rate]
Consider the trace process on $\{\bf1,\bf2,\bf3\}$. Let $r_\beta(\bf i, \bf j)$ be the transition rate from $\bf i$ to $\bf j$ of the trace process. Then the scaled jump rate $\theta_\beta^1 r_\beta(\bf i, \bf j)$ converges to
\[
    \limb \theta_\beta^1 r_\beta(\bf i, \bf j) = \begin{cases}
        1 & (\bf i, \bf j) = (\bf1, \bf3) \\
        0 & otherwise.
    \end{cases}
\]
\end{proposition}
\begin{proof}
    From \cite[Lemma 6.8]{Beltr_n_2010}, we have
    \begin{align*}
        \mu_\beta(\bf1)r_\beta(\bf1, \bf2) = \frac{1}{2}[\capa (\bf1, \bf2 \cup \bf3) + \capa(\bf2, \bf1 \cup \bf3) - \capa(\bf3, \bf1 \cup \bf2 )] \\
        \mu_\beta(\bf1)r_\beta(\bf1, \bf3) = \frac{1}{2}[\capa (\bf1, \bf2 \cup \bf3) + \capa(\bf3, \bf1 \cup \bf2) - \capa(\bf2, \bf1 \cup \bf3 )]
    \end{align*}
    and
    \[
        \frac{r_\beta(\bf1, \bf2)}{r_\beta(\bf1, \bf3)} = \frac{\PP_{\bf1}^\beta[\tau_{\bf2} < \tau_{\bf3}]}{\PP_{\bf1}^\beta[\tau_{\bf3} < \tau_{\bf2}]}.
    \]
    This gives
    \begin{align*}
        e^{\beta\Gamma^\star_{13}} r_\beta(\bf1, \bf3) & = e^{\beta\Gamma^\star_{13}}\PP_{\bf1}^\beta[\tau_{\bf3} < \tau_{\bf2}](r_\beta(\bf1, \bf2) + r_\beta(\bf1, \bf3)) \\
        & = e^{\beta\Gamma^\star_{13}}\PP_{\bf1}^\beta[\tau_{\bf3} < \tau_{\bf2}] \frac{\capa(\bf1, \bf2 \cup \bf3)}{\mu_\beta(\bf1)}
    \end{align*}
    Together with Propositions \ref{prop:EKcap} and \ref{prop:pot_zero}, this converges to
    \[
        \kappa_1\PP_{\bf1}^\beta[\tau_3 < \tau_2] (1 + o_\beta(1)) \to \kappa_1
    \]
    as $\beta \to \infty$. Similarly, as $\beta \to \infty$,
    \begin{align*}
        e^{\beta\Gamma^\star_{13}}r_\beta(\bf1, \bf2) 
        & = e^{\beta\Gamma^\star_{13}}\frac{1}{\mu(\bf1)}\PP_{\bf1}^\beta[\tau_{\bf2} < \tau_{\bf3}] \capa(\bf1, \bf2 \cup \bf3) \\
        & = \kappa_1 \PP_{\bf1}^\beta[\tau_{\bf3} < \tau_{\bf2}](1 + o_\beta(1)) \to 0
    \end{align*}

    On the other hand, \cite[Lemma 6.7]{Beltr_n_2010} gives that as $\beta \to \infty$,
    \begin{align*}
        e^{\beta\Gamma^\star_{13}}r_\beta(\bf2, \bf3) & = e^{\beta\Gamma^\star_{13}}\frac{\capa(\bf2, \bf3)}{\mu(\bf2)} = \kappa_2 e^{\beta(\Gamma^\star_{13} - \Gamma^\star_{23})}(1 + o_\beta(1)) \to 0, \text{ and }
    \end{align*}
    \begin{align*}
        e^{\beta\Gamma^\star_{13}}r_\beta(\bf2, \bf1) & = e^{\beta\Gamma^\star_{13}}\frac{\capa(\bf1, \bf2)}{\mu(\bf2)} \le C e^{\beta(\Gamma^\star_{13}+\Phi(\bf1, \bf2) - H(\bf2))} \to 0.
    \end{align*}
\end{proof}

\begin{proposition}[Theorem \ref{thm:MCMR2}: Convergence of the rate]
Consider a trace process on $\{\bf2,\bf3\}$. Let $r_\beta(\bf i, \bf j)$ be the transition rate from $\bf i$ to $\bf j$ of the trace process. Then the scaled jump rate $\theta_\beta^2 r_\beta(\bf i, \bf j)$ converges to
\[
    \limb \theta_\beta^2 r_\beta(\bf2, \bf3) = 1, \quad \limb \theta_\beta^1 r_\beta(\bf3, \bf2) = 0.
\]
\end{proposition}
\begin{proof}
    It follows from \cite[Lemma 6.7]{Beltr_n_2010} that as $\beta \to \infty$, we have 
    \begin{align*}
        e^{\beta\Gamma^\star_{23}}r_\beta(\bf2, \bf3) = e^{\beta\Gamma^\star_{23}}\frac{\capa(\bf2, \bf3)}{\mu_\beta(\bf2)} \to \kappa_2 \text{ and } e^{\beta\Gamma^\star_{23}}r_\beta(\bf3, \bf2) = e^{\beta\Gamma^\star_{23}}\frac{\capa(\bf3, \bf2)}{\mu_\beta(\bf3)} \to 0.
    \end{align*}
\end{proof}

To finish the proof of Theorems \ref{thm:MCMR1} and \ref{thm:MCMR2}, it remains to show the negligibility.
We propose a general form of the statement.
\begin{proposition}
    \label{prop:neg}
    Suppose $\mc M \subseteq \XX$ consists of states with stability levels strictly greater than states not in $\mc M$, that is,
    \[
        \mc V_\sigma < \mc V_i \quad \text{ for all } i \in \mc M, \, \sigma \notin \mc M.
    \]
    Let $\mc V = \min_{i \in \mc M} \mc V_i$ and $\tb = e^{\beta \mc V}$. Denote $\Delta = \XX \backslash \mc M$. Then we have
    \[
        \limb \frac{1}{\tb}\EE_\xi^\beta \left[\int_0^{t\tb} \one_{\{X_\beta(s)\in \Delta\}}ds \right] = 0
    \]
    for every $\xi \in \XX$ and $t > 0$. In other words, the chain speeded up by $\tb$ spends a negligible amount of time outside of $\mc M$.
\end{proposition}
\begin{proof}
    \,
    We first prove the statement for $\xi$ in $\mc M$, then extend the result to $\xi$ that are not in $\mc M$. This will be based on the observation that the state not in $\mc M$ falls to $\mc M$ in a negligible time.

    \vspace{5pt}
    \noindent
    \emph{Step 1.} The statement holds for $\xi \in \mc M$.

    We prove the statement inductively, in an energy-increasing order. First, let $\xi \in \mc M$ be a state with minimal energy. Let $\PP_{\mu_\beta}$ denote the law of the Metropolis dynamics starting from the initial distribution $\mu_\beta$. Then, for any time $s > 0$, we have
    \[
        \PP_\xi^\beta[X_\beta(s) \in \Delta] \le \frac{1}{\mu_\beta(\xi)}\PP_{\mu_\beta}[X_\beta(s) \in \Delta] = \frac{\mu_\beta(\Delta)}{\mu_\beta(\xi)}.
    \]
    By Fubini's theorem, we have
    \begin{align*}
        \frac{1}{\tb}\EE_\xi^\beta \left[\int_0^{t\tb} \one_{\{X_\beta(s)\in \Delta\}}ds \right] = \frac{1}{\tb}\int_0^{t\tb} \PP_\xi^\beta \left[X_\beta(s) \in \Delta\right]
        \le t \cdot \frac{\mu_\beta(\Delta)}{\mu_\beta(\xi)} = o_\beta(1).
    \end{align*}
    Hence, the statement holds for the state with minimal energy.

    Next, we take $\xi \in \mc M$, and define $\mc M' = \{\sigma \in \mc M :\, H(\sigma) < H(\xi)\}$. By the induction hypothesis, the statement holds for states in $\mc M'$. Since $\mc V_\xi \ge \mc V$ by assumption, we divide the case into two.

    \vspace{5pt}
    \noindent
    \textbf{(Case 1: $\mc V_\xi > \mc V$)} Let $C$ be the cycle of depth $\mc V_\xi$ containing $\xi$ as a bottom. By \cite[Theorem 3.2]{Nardi_2015}, $\PP_\xi[\tau_{\partial C} < t\tb] = o_\beta(1)$. So we may assume that the dynamics take place in $C$ on time $[0, t\tb]$. By the same argument as the previous case, Fubini's theorem gives
    \begin{align*}
        \frac{1}{\tb}\EE_\xi^\beta \left[\int_0^{t\tb} \one_{\{X_\beta(s)\in \Delta\}}ds ; \tau_{\partial C} > t \tb\right] & = \frac{1}{\tb}\int_0^{t\tb} \PP_\xi^\beta \left[X_\beta(s) \in \Delta, \, \tau_{\partial C} > t \tb \right] \\
        & \le t \cdot \frac{\mu_\beta(\Delta \cap C)}{\mu_\beta(\xi)} = o_\beta(1)
    \end{align*}

    \vspace{5pt}
    \noindent
    \textbf{(Case 2: $\mc V_\xi = \mc V$)} Putting $g = \one_\Delta$ in Proposition \ref{prop:EKgen}, we have
    \[
        \frac{1}{\tb}\EE_\xi^\beta\left[\int_0^{\tau_{\mc M'}} \one_{\{X_\beta(s) \in \Delta\}} ds\right] = \frac{1}{\tb\capa(\xi, \mc M')}\sum_{\sigma \in \Delta} \PP_\sigma^\beta[\tau_\xi < \tau_{\mc M'}] \mu_\beta(\sigma).
    \]
    Define $\mc V' = \max_{\sigma \notin M} \mc V_\sigma$, which is less than $\mc V$ by assumption. It is clear that $H(\xi) + \mc V_\xi \le \Phi(\xi, \mc M')$. Moreover, retaining the energy lower than $H(\xi) + \mc V_\xi$, we may move from $\xi$ to a state with lower energy. Repeat this until the process hits $\mc M'$. Every state visited must have a stability level at most $\mc V' < \mc V_\xi$, so we have that $\Phi(\xi, \mc M') \le H(\xi) + \mc V_\xi$. Hence, $\Phi(\xi, \mc M') = H(\xi) + \mc V_\xi$ and equivalently, $\Gamma(\xi, \mc M') = \mc V_\xi$. Together with Lemma \ref{lem:cap_est}, we have
    \[
        \capa(\xi, \mc M') \ge C e^{-\beta\Phi(\xi, \mc M')}Z_\beta^{-1} = C e^{-\beta(\mc V_\xi + H(\xi))}Z_\beta^{-1} = C e^{-\beta \mc V_\xi}\mu_\beta(\xi)
    \]
    for some $C > 0$. This gives 
    \[
        \frac{1}{\tb}\EE_\xi^\beta\left[\int_0^{\tau_{\mc M'}} \one_{\{X_\beta(s) \in \Delta\}} ds\right] \le C^{-1}\sum_{\sigma \in \Delta} \PP_\sigma^\beta[\tau_\xi < \tau_{\mc M'}]\frac{\mu_\beta(\sigma)}{\mu_\beta(\xi)}.
    \]
    We divide the sum over $\sigma \in \Delta$ into two parts and show that each term vanishes as $\beta \to 0$.
    \begin{enumerate}[leftmargin=*]
        \item $H(\sigma) > H(\xi)$: Bounding the probability by 1, the term exponentially vanishes.
        \item $H(\sigma) \le H(\xi)$: Starting from $\sigma$, we may reach a state with lower energy without exceeding $H(\sigma) + \mc V'$. Repeating this until we reach $\mc M'$, we have that $\Gamma(\sigma, \mc M') \le \mc V'$. Also, observe that $\Gamma(\xi, \mc M') \le \max(\Gamma(\xi, \sigma), \Gamma(\sigma, \mc M'))$ and
        \[
            \Gamma(\sigma, \mc M') \le \mc V' < \mc V_\xi = \Gamma(\xi, \mc M').
        \]
        Hence $\Gamma(\xi, \mc M') \le \Gamma(\xi, \sigma)$ and $\Gamma(\sigma, \mc M') < \Gamma(\xi, \sigma)$.
        By Lemma \ref{lem:pot_est} on $\PP_\sigma^\beta[\tau_\xi < \tau_{\mc M'}]$,
        \begin{align*}
            \PP_\sigma[\tau_\xi < \tau_{\mc M}]\frac{\mu_\beta(\sigma)}{\mu_\beta(\xi)} \le \frac{\capa(\sigma, \xi)e^{\beta H(\xi)}}{\capa(\sigma, \mc M')e^{\beta H(\sigma)}},
        \end{align*}
        and by Lemma \ref{lem:cap_est} on both capacities, the above term is bounded by
        \[
            C e^{-\beta(\Phi(\xi, \sigma) - H(\xi))+\beta(\Phi(\sigma, \mc M')-H(\sigma))} 
            = C e^{-\beta(\Gamma(\xi, \sigma) - \Gamma(\sigma, \mc M'))} = o_\beta(1).
        \]
        for some constant $C > 0$.
    \end{enumerate}
    Gathering two bounds, we obtain
    \[
        \frac{1}{\tb}\EE_\xi^\beta\left[\int_0^{\tau_{\mc M'}} \one_{\{X_\beta(s) \in \Delta\}} ds\right] = o_\beta(1).
    \]
    Finally, we bound the integral over $[0, t\tb]$. In the case of $\tau_{\mc M'} > t\theta_\beta$, the previous result gives
    \[
        \frac{1}{\tb}\EE_\xi^\beta\left[\int_0^{t\tb} \one_{\{X_\beta(s) \in \Delta\}} ds; \, \tau_{\mc M'} > t\tb \right] \le \frac{1}{\tb}\EE_\xi^\beta\left[\int_0^{\tau_{\mc M'}} \one_{\{X_\beta(s) \in \Delta\}} ds\right] = o_\beta(1).
    \]
    In the case of $\tau_{\mc M'} > t \tb$, we split the integral over $[0, \tau_{\mc M'}]$ and $[\tau_{\mc M'}, \tb]$. The former integral is again $o_\beta(1)$ by the previous result. The latter integral can be written in
    \begin{align*}
        \frac{1}{\tb}\EE_\xi^\beta\left[\int_{\tau_{\mc M'}}^{t\tb} \one_{\{X_\beta(s) \in \Delta\}} ds; \, \tau_{\mc M'} < \tb \right] \le \sum_{\sigma \in \mc M'} \frac{1}{\tb} \PP_{\xi}^\beta[\tau_{\mc M'} = \tau_\sigma ] \EE_\sigma^\beta \left[\int_0^{t\tb} \one_{\{X_\beta(s) \in \Delta\}} ds \right].
    \end{align*}
    by strong Markov property on $\tau_{\mc M'}$. Applying the induction hypothesis to each $\sigma \in \mc M'$, the summation is bounded by $o_\beta(1)$. Summing up the results, we obtain
    \[
        \frac{1}{t\tb}\EE_\xi^\beta\left[\int_0^{\tau_{\mc M'}} \one_{\{X_\beta(s) \in \Delta\}} ds\right] = o_\beta(1),
    \]
    completing the proof of the statement for $\xi \in \mc M$.

    \vspace{5pt}
    \noindent
    \emph{Step 2}. The statement holds for $\xi \notin \mc M$.
    
    To see this, consider a maximal cycle consisting of $\XX \backslash \mc M$. Since the stability level is bounded by $\mc V'$, the depths of cycles are bounded by $\mc V'$. By \cite[Proposition 3.7]{Nardi_2015}, for sufficiently small $\ee > 0$
    \[
        \PP_\xi^\beta\left[\tau_{\mc M} > t\tb\right] < \PP_\xi^\beta\left[\tau_{\mc M} > e^{\beta(\mc V' + \ee)} \right] = o_\beta(1).
    \]
    We split the integral into two parts according to $\tau_{\mc M}$ and $t\tb$. The first part becomes
    \[
        \frac{1}{\tb}\EE_\xi^\beta \left[\int_0^{t\tb} \one_{\{X_\beta(s)\in \Delta\}}ds;\, \tau_{\mc M} > t\tb \right] \le t \PP_\xi^\beta[\tau_{\mc M} > t\tb] = o_\beta(1).
    \]
    For the second part, we use the strong Markov property on $\tau_{\mc M}$. Then
    \[
        \frac{1}{\tb}\EE_\xi^\beta \left[\int_0^{t\tb} \one_{\{X_\beta(s)\in \Delta\}}ds;\, \tau_{\mc M} < t\tb \right] \le \frac{1}{\tb}\sum_{m \in \mc M} \PP_\xi^\beta[\tau_{\mc M} = \tau_m] \EE_m^\beta \left[\int_0^{t\tb} \one_{\{X_\beta(s)\in \Delta\}}ds \right].
    \]
    The desired result follows from the case where $\xi \in \mc M$.
\end{proof}

\subsection{Computation of transition time}
\label{subsec:4.4}

Recall from Proposition \ref{prop:EKcap} the equality
\begin{equation}\label{eq:EK13}
    \EE_1^\beta[\tau_3] = \frac{1}{\capa(\bf1, \bf3)} \sum_{\sigma \in \XX} \mu_\beta(\sigma) \PP_\sigma^\beta[\tau_1 < \tau_3].
\end{equation}
First, we show that states other than \textbf{1}, \textbf{2}, and \textbf{3} do not contribute to the summation. The general form of this statement is given in  \cite[Lemma 3.3]{Landim_2016}, which includes our case as shown in the next lemma.
\begin{lemma}
    \label{lem:ekdom}
    Suppose $\mc M \subseteq \XX$ consists of states with stability levels strictly greater than states not in $\mc M$, which is
    \[
       \mc V_\sigma < \mc V_\eta \quad \text{ for all } \eta \in \mc M, \, \sigma \notin \mc M.
    \]
    Then, for any $B \subset \mc M$ and $\xi \in \mc M \backslash B$,
    \[
        \EE_\xi^\beta[\tau_B] = (1 + o_\beta(1))\frac{1}{\capa(\xi, B)} \sum_{\sigma \in \mc M} \mu(\sigma) \PP_\sigma^\beta[\tau_\xi < \tau_B].
    \]
\end{lemma}
\begin{proof}
    For a positive functions $f(\beta)$ and $g(\beta)$ on $\beta$, we use the notation $f(\beta) \ll g(\beta)$ if $\limb f(\beta)/g(\beta) = 0$. We define $f(\beta) \gg g(\beta)$ in a similar fashion. By \cite[Lemma 3.3]{Landim_2016}, it suffices to check that for $\sigma \notin \mc M$, $\eta \in \mc M$ such that $H(\sigma) \le H(\eta)$,
    \[
        \frac{\capa(\sigma, \eta)}{\mu(\eta)} \ll \frac{\capa(\sigma, \mc M)}{\mu(\sigma)},
    \]
    or equivalently $\Gamma(\eta, \sigma) > \Gamma(\sigma, \mc M)$. If $H(\eta) < H(\sigma)$, then $\Gamma(\eta, \sigma) \ge \mc V_\eta$ by definition, and $\Gamma(\sigma, \mc M) \le \max_{i \in \XX \backslash \mc M} \mc V_i$ by starting from $\sigma$ and repeatedly moving to a state with lower energy until hitting $\mc M$. Hence, the desired condition holds. If $H(\eta) = H(\sigma)$, then the energy barrier $\Gamma(\eta, \sigma)$ must be greater than $\mc V_\eta$, or unless $\mc V_\sigma = \mc V_\eta$ and this contradicts $\sigma \notin \mc M$ and $\eta \in \mc M$. Hence $\Gamma(\eta, \sigma) > \mc V_\eta \ge \max_{\sigma' \in \XX \backslash \mc M} \mc V_{\sigma'} \ge \mc V_\sigma$ as before. Therefore, the desired condition holds for either case.
\end{proof}

Finally, we prove Theorem \ref{thm:EKinf}. Applying the previous lemma to \eqref{eq:EK13}, we have
\[
    \EE_{\bf1}^\beta[\tau_{\bf3}] = (1 + o_\beta(1))\frac{1}{\capa(\bf1, \bf3)} \left( \mu_\beta(\bf1) + \mu_\beta(\bf2) \PP_{\bf2}^\beta[\tau_{\bf1} < \tau_{\bf3}] \right).
\]
Putting the estimate of $\capa(\bf1, \bf3)$ from Proposition \ref{prop:cap},
\[
    e^{-\beta\Gamma_{13}^\star}\EE_{\bf1}^\beta[\tau_{\bf3}] = (\kappa_1 + o_\beta(1)) \left(1 + e^{\beta(H(\bf1) - H(\bf2))} \PP_{\bf2}^\beta[\tau_{\bf1} < \tau_{\bf3}]\right).
\]
Hence, it suffices to obtain a sufficient lower bound of $\PP_{\bf2}^\beta[\tau_{\bf1} < \tau_{\bf3}]$. This can be done by constructing a cycle-path from \bf2 to \bf1 disjoint to \bf3, then computing the probability of following the path. Readers can find a similar construction of paths in \cite[Lemma 6.1]{Landim_2016}, for example.

\begin{proof}[Proof of Theorem \ref{thm:EKinf}]
    To obtain a lower bound of $\PP_{\bf2}^\beta[\tau_{\bf1} < \tau_{\bf3}]$, we construct a cycle-path from \bf2 to \bf1 disjoint to \bf3. This is done by reversing a reference path of the Ising model of spin 1 and 2. Precisely, we divide the path into the following four parts.
    \begin{itemize}[leftmargin=*]
        \item Let $\gamma_0$ be a path from \bf2 to $\rect{2}{1}{K-1}{L-1}$ obtained by adding 1-spin along one row, then adding 1-spin along one column. We see that the energy strictly increases.
        \item For $1 \le i \le K - L-1$, let $\gamma_i$ be a path from $\rect{2}{1}{K-i}{L-1}$ to $\rect{2}{1}{K-i-1}{L-1}$ obtained by consecutively adding 1-spin along a vertical side of 2-cluster.
        \item For $K-L \le i \le K+L-1$, let $\gamma_i$ be a path adding a line of 1-spin to transform a 2-cluster of size $\ell \times \ell$ into size $\ell \times (\ell - 1)$, or 2-cluster of size $(\ell+1)\times\ell$ into size $\ell \times \ell$, alternatively. We end at $\rect{2}{1}{\ell_c^{12}}{\ell_c^{12}}$. Finally, let $\gamma_{K+L-1+2\ell_c^{12}+1}$ be a path connecting $\rect{2}{1}{\ell_c^{12}}{\ell_c^{12}}$ and $\dect{2}{1}{\ell_c^{12} - 1}{\ell_c^{12}}$, which is the saddle configuration of the Ising model between \bf1 and \bf2. 
        \item Let $\beta_0$ be a length 2 path from $\dect{2}{1}{\ell_c^{12} - 1}{\ell_c^{12}}$ to $\rect{2}{1}{\ell_c^{12} - 1}{\ell_c^{12}}$. For $1 \le i \le 2\ell_c^{12}$, let $\beta_i$ be a path defined the same as above, connecting $\rect{2}{1}{\ell_c^{12}}{\ell_c^{12}}$ and $\rect{2}{1}{0}{0}$, which is \bf1.
    \end{itemize}
    Let $\gamma$ be a path obtained by concatenating $\gamma_0, \ldots, \gamma_{K+L-1+2\ell_c^{12}+1}$, and $\beta$ be a path concatenating $\beta_0, \ldots, \beta_{2\ell_c^{12}}$. Since the energy fluctuates along the path $\gamma \cup \beta$, computing the transition rate along the path does not yield a desirable lower bound. Nevertheless, we may remedy this by considering a cycle-path and then computing the probability of the process following the cycle-path. Namely, for each state in $\gamma \cup \beta$, consider a maximal cycle disjoint from \bf3. This gives a cycle-path from $C(\bf2)$ to $C(\bf1)$. Here, $C(\sigma)$ denotes the maximal cycle disjoint from \bf3 and containing $\sigma$.
    We claim that the energy increases along the cycle-path of $\gamma$, and that energy decreases along the cycle-path of $\beta$.

    First, we prove the claim on each path $\gamma_k$. Denote the path as $\gamma_k = (\eta_0, \ldots, \eta_n)$. If $k = 0$, the energy of the state strictly increases, so we are done. Thus, we may assume $k > 0$. Then $\eta_0 \in \rect{2}{1}{n}{m+1}$ and $\eta_n \in \rect{2}{1}{n}{m}$ for $n, m \ge \ell_c^{12}$.
    
    From $\eta_0$ to $\eta_{n-1}$, the energy increase by $h_2 - h_1$. To compare $\eta_{n-1}$ and $\eta_n$, recall the variant of Lemma \ref{lem:cycle} on \bf1 and \bf2 that $\eta_{n-1}$ is a state with minimal energy in $\partial C(\eta_n)$. Hence, the energy monotone increases along the cycle-path obtained by $(\eta_0, \ldots, \eta_n)$. Therefore, we have the claim for $\gamma$.

    We may prove the claim similarly on each path $\beta_k$. In this case, each $\beta_k$ is contained in a single cycle. These cycles are adjacent and the energy decreases along $\beta$.

    Finally, we compute the probability of the process following the cycle-path constructed above. Let $(C_1, \ldots ,C_a)$ be a uphill cycle-path from $C_1 = C(\bf2)$ to $C_a = C(\dect{2}{1}{{\ell_c^{12}-1}}{\ell_c^{12}})$, and $(D_1, \ldots, D_b)$ be a downhill cycle-path from $D_1 = C_a$ to $D_b = C(\bf1)$. Note that this cycle-path is disjoint from \bf3. We have
    \begin{align*}
        \PP_{\bf2}[\tau_{\bf1} < \tau_{\bf3}] \ge \PP_2^\beta[\tau_{\partial C_0} = \tau_{C_1}] \prod_{i = 1}^{a-1} \sup_{\sigma \in C_i}\PP_\sigma^\beta[\tau_{\partial C_i} = \tau_{C_{i+1}}] \prod_{j=1}^{b-1} \sup_{\sigma \in D_i} \PP_\sigma^\beta[\tau_{\partial D_i} = \tau_{D_{i+1}}]. 
    \end{align*}
    From Proposition \ref{prop:exit}, for any $\ee > 0$, for all $\beta$ sufficiently large,
    \[
        \sup_{\sigma C_i} \PP_\sigma^\beta[\tau_{\partial C_i} = \tau_{C_{i+1}}] \ge e^{-\beta[H(\FF(\partial C_{i+1})) - H(\FF(\partial C_i))+\ee]_+}.
    \]
    Therefore, we have
    \begin{align*}
        \PP_{\bf2}^\beta[\tau_{\bf1} < \tau_{\bf3}] & \ge e^{-\beta[H(\dect{2}{1}{\ell_c^{13}+1}{\ell_c^{13}}) - H(C(\bf2))+\ee]} \\
        & = e^{-\beta[(H(\bf1) + \Gamma^\star_{12}) - (H(\bf2) + \Gamma^\star_{23})+\ee]}.
    \end{align*}
    Assuming $\Gamma^\star_{12} > \Gamma^\star_{23}$, we have
    \[
        \mu(2)\PP_2^\beta[\tau_1 < \tau_3] \gg \mu(1),
    \]
    so the left-hand-side probability majorizes the summation \eqref{eq:EK13}. Hence,
    \[
        \limb e^{-\beta\Gamma^\star_{13}}\EE_1^\beta[\tau_3] = \infty.
    \]
\end{proof}

\appendix
\section{Appendix}
\subsection{Reference path of the Ising model}
\label{app:a1}

Recall the Ising model on a lattice graph $\Lambda = (V, E)$ defined by a Hamiltonian function
\[
    H(\sigma) = -J \sum_{(x,y) \in E} \one_{\{\sigma(x) = \sigma(y)\}} - h \sum_{x \in V} \one_{\{\sigma(x) = 1\}}
\]
where $\sigma \in \{-1, 1\}^V$ is spin configuration on $\Lambda$ and $J, h > 0$. For simplicity, we put $J = 1$. The model has a unique metastable state $\bf{-1}$, the configuration with constant spin $-1$, and a unique stable state $\bf{+1}$, the configuration with constant spin $1$.

We describe the well-known \emph{reference path}, which is the energy minimizing path from $\bf{-1}$ to $\bf{+1}$. While it was originally introduced in \cite{neves1991critical}, it has become a well-established concept and is explained in various literature. As an example, one can find a detailed explanation in \cite[Section 17]{bovier2015metastability}.
\begin{itemize}[leftmargin=*]
    \item Starting from $\bf{-1}$, flip a single $-1$ spin to make a 1-cluster, a connected component of 1-spin, of size $1 \times 1$.
    \item Having an 1-cluster of size $i \times i$ or $i \times (i+1)$ in a sea of $-1$-spins, flip a $-1$-spin adjacent to an edge of length $i$. This process increases the energy by $2 - h$. Accordingly, flip the rest of the $-1$-spins into $1$-spins along the edge. We end up having an 1-cluster of size $i \times (i+1)$ or $(i+1) \times (i+1)$, respectively. Note that each step decreases the energy by $-h$.
    \item Repeat the previous step until the 1-cluster grows up to the whole lattice.
\end{itemize}
Along the reference path, the Hamiltonian function has a unique maximum when adding a spin to a 1-cluster of size $(\ell_c-1) \times \ell_c$. This configuration is called the \emph{critical droplet}. Moreover, the direct computation gives a energy barrier from $\bf{-1}$ to $\bf{+1}$, written by
\[
    f(h) \coloneqq 4\ell_c - h(\ell_c(\ell_c - 1) + 1), \quad \ell_c = \left\lceil \frac{2}{h} \right\rceil.
\]

\subsection{Auxiliary function}
\label{app:a2}
In this section, we provide a comparison between several quantities appearing throughout the paper. Let $f(h)$ be a function on $h$ defined by the above equation. Recall from \eqref{eq:crit} the quantities 
\[
    \Gamma_{ij}^\star = f(h_j - h_i),\quad \Phi_{ij}^\star = f(h_j - h_i) + H(\bf i)
\]
which were used to represent the energy barrier and communication height of our model.
\begin{lemma}
    A function $f(h)$ is continuous and strictly decreasing on $h \in (0, \infty)$. Moreover, $5 < f(h) < \frac{8}{h}$ for $h \in (0, 1)$.
\end{lemma}
\begin{proof}
    Let $h \in [\frac{2}{m+1}, \frac{2}{m})$ for $m \in \ZZ^+$. Then $\ell_c = m+1$ and we have
    \[
        f(h) = 4(m+1) - h(m(m+1)+1),
    \]
    which is a continuous decreasing function on $[\frac{2}{m+1}, \frac{2}{m})$. Also, note that
    \[
        \lim_{h \to \frac{2}{m}+} f(h) = 4m - \frac{2}{m}\left(m(m-1)+1\right) = 2m + 2 - \frac2m
    \]
    and
    \[
        \lim_{h\to \frac{2}{m}-} f(h) = 4(m+1) - \frac{2}{m}\left(m(m+1)+1\right) = 2m + 2 - \frac2m.
    \]
    Therefore, $f(h)$ is continuous and strictly decreasing on $(0, \infty)$.

    Finally, we have that
    \[
        f(h) > f(1) = 5, \quad f(h) = 4m - h(m(m-1)+1) < 4m < \frac{8}{h}
    \]
    for $h \in (0, 1)$. This proves the lemma.
\end{proof}

The lemma directly implies that $\Gamma_{13}^\star < \Gamma_{23}^\star$ and $\Gamma_{13}^\star < \Gamma_{12}^\star$. Moreover,
\[
    \Gamma_{13}^\star = f(h_3 - h_1) > 5
\]
by the assumption $0 < h_1 < h_2 < h_3 < 1$ in Section \ref{sec:model}.

Finally, we compare the Hamiltonian between states of our interest: three monochromatic configurations \bf1, \bf2, and \bf3, and saddle configurations between each two of these.
\begin{lemma}
\label{lem:app2}
The following inequality holds.
\[
    \Phi_{12}^\star > \Phi_{13}^\star > H(\bf 1) > \Phi_{23}^\star > H(\bf2) > H(\bf3).
\]
\end{lemma}
We also remark that the lemma shows that the six states of Figure \ref{fig:energy} are depicted in the correct order of energy level.
\begin{proof}
    Recall that $K \times L$ is the size of the lattice graph the model lives. We have assumed in Assumption A that $K, L > \frac{3}{h_3-h_2}, \, \frac{3}{h_2-h_1}$.

    The first three quantities can be compared using the previous lemma. Indeed, $h_2 - h_1 < h_3 - h_1 < 1$ implies
    \[
        \Gamma_{12}^\star = f(h_2 - h_1) > \Gamma_{13}^\star = f(h_3 - h_1) > 0.
    \]
    Adding $H(\bf1)$ to each side gives the desired inequality.

    Next, we show $H(\bf1) > \Phi_{23}^\star$. This is equivalent to showing
    \[
        (h_2 - h_1)KL > \Gamma_{23}^\star = f(h_3 - h_2).
    \]
    By Assumption A, we have
    \[
        (h_2 - h_1)KL > (h_2 - h_1) \cdot \frac{3}{h_2-h_1} \cdot \frac{3}{h_3 - h_2} > f(h_3 - h_2)
    \]
    where the last inequality follows from the previous lemma. Hence, we have $H(\bf1) > \Phi_{23}^\star$.

    Finally, $\Phi_{23}^\star > H(\bf2)$ follows from $\Gamma_{23}^\star > 0$, and clearly $H(\bf2) > H(\bf3)$. This completes the proof.
\end{proof}

\paragraph{Acknowledgement}
This research is supported by the National Research Foundation of Korea (NRF) grant funded by the Korea government (MSIT) (No. 2023R1A2C100517311). The author thanks Insuk Seo and Seonwoo Kim for introducing the problem and offering enlightening insights. 

\bibliographystyle{abbrv}
\bibliography{reference}

\end{document}